\documentclass[ a4paper, 12pt, reqno ]{amsart}
\usepackage{ amsmath, amssymb, amsthm, fancyhdr, enumerate, enumitem, tabu, units }
\usepackage{geometry}[margin=1in]
\usepackage[sortcites,maxbibnames=10,block=ragged]{biblatex}
\usepackage{xfrac}
\usepackage{tikz}
\usepackage{xcolor}
\usepackage{lscape}
\usepackage{colonequals}
\usepackage{url}
\usepackage{ragged2e}

\newcommand{\NN}{\mathbb{N}}

\newcommand{\QQ}{\mathbb{Q}}
\newcommand{\ZZ}{\mathbb{Z}}
\newcommand{\CC}{\mathbb{C}}

\newtheorem{theorem}{Theorem}[section]
\newtheorem{conjecture}[theorem]{Conjecture}
\newtheorem{corollary}[theorem]{Corollary}

\newtheorem{lemma}[theorem]{Lemma}

\newtheorem{prop}[theorem]{Proposition}

\theoremstyle{definition}

\newtheorem{remark}{Remark}

\newenvironment{singnumalign}{
    \begin{equation}
    \begin{aligned}
}{
    \end{aligned}
    \end{equation}
    \ignorespacesafterend
}

\title{On Some Hypergeometric Supercongruence Conjectures of Long}
\author{Michael Allen}
\email{allenm3@oregonstate.edu}
\address{368 Kidder Hall, Oregon State University, Corvallis, OR 97331}

\begin{document}

\maketitle


\begin{abstract}
    In 2003, Rodriguez Villegas conjectured 14 supercongruences between hypergeometric functions arising as periods of certain families of rigid Calabi-Yau threefolds and the Fourier coefficients of weight 4 modular forms.  Uniform proofs of these supercongruences were given in 2019 by Long, Tu, Yui, and Zudilin.  Using p-adic techniques of Dwork, they reduce the original supercongruences to related congruences which involve only the hypergeometric series.  We generalize their techniques to consider six further supercongruences recently conjectured by Long.  In particular we prove an analogous version of Long, Tu, Yui, and Zudilin's reduced congruences for each of these six cases.  We also conjecture a generalization of Dwork's work which has been observed computationally and which would, together with a proof of modularity for Galois representations associated to our hypergeometric data, yield a full proof of Long's conjectures.
\end{abstract}
%
\section{Introduction and Statement of Results}
Hypergeometric functions have long been studied throughout many different branches of mathematics.  In recent years, there has been significant progress in studying hypergeometric functions in relation to automorphic forms.  Finite field analogs of hypergeometric functions were introduced by Greene in 1987 \cite{Greene}, and several authors have since related point counts on abelian varieties to these finite field hypergeometric functions \cite{Greene, Ono98, BCM, FLRST}.  Values of truncated hypergeometric series are often congruent modulo primes $p$ to various objects of mathematical interests.  In rarer instances, these congruences hold modulo larger than expected powers of $p$, in which case we refer to them as supercongruences.  Supercongruences have been established relating hypergeometric functions to Ramanujan-Sato series \cite{Long11, MO, Mortenson, Swisher, vH, Zudilin}, to Ap\'{e}ry numbers and similar sequences \cite{OSZ, AO}, and to modular forms \cite{McCarthy, LTYZ, Kilbourn, FM}.  In this paper, we consider the latter type of supercongruence.  In particular, we investigate a number of supercongruences between truncated hypergeometric series and Fourier coefficients of modular forms which were recently conjectured by Long \cite{Long20}. \par
We consider the hypergeometric datum given as the tuple $(\alpha, \beta, \lambda)$, where $\lambda$ is a nonzero element of an abelian extension of $\QQ$, and $\alpha = \left\{ r_1, \cdots, r_n\right\}$ and $\beta = \left\{1, q_1, \hdots, q_{n-1}\right\}$ are multisets of rationals satisfying $r_i - q_j \not\in \ZZ$ for all $i$ and $j$.  We define the classical (generalized) hypergeometric series $F(\mathbf{\alpha}, \mathbf{\beta}; \lambda)$ by
\[
    F(\mathbf{\alpha}, \mathbf{\beta}, \lambda) \colonequals {}_nF_{n-1} \left[\begin{matrix} r_1 & r_2 & \cdots & r_n \\ & q_1 & \cdots & q_{n-1} \end{matrix} \hspace{1ex};\hspace{1ex} \lambda \right] 
    = \sum_{k=0}^\infty \frac{(r_1)_k(r_2)_k\hdots(r_n)_k}{(q_1)_k\cdots(q_{n-1})_k} \frac{\lambda^k}{k!},
\]
where $(a)_k$ denotes the rising factorial or Pochhammer symbol 
\[ 
    (a)_k \colonequals \frac{\Gamma(a+k)}{\Gamma(a)} = a(a+1)(a+2)\cdots(a+k-1),
\]
and we take the convention $(a)_0=1$.  In particular, $(1)_k = k!$.
Given a prime $p$ and a positive integer $s$, it will be convenient for us to denote the truncation of our hypergeometric series at $p^s-1$ and evaluated at $1$ by
\[
    F_s(\mathbf{\alpha}, \mathbf{\beta}) \colonequals {}_nF_{n-1} \left[\begin{matrix} r_1 & r_2 & \cdots & r_n \\ & q_1 & \cdots & q_{n-1} \end{matrix} \hspace{1ex};\hspace{1ex} 1 \right]_{p^s-1} = \sum_{k=0}^{p^s-1} \frac{(r_1)_k(r_2)_k\hdots(r_n)_k}{(q_1)_k(q_2)_k\cdots(q_{n-1})_k(1)_k}
\]
Note that we take the non-standard convention to define $\beta$ as $\left\{1, q_1, \hdots, q_{n-1}\right\}$ rather than $\left\{q_1, \hdots, q_{n-1}\right\}$.  This choice allows us to more easily consider the additional $k!$ appearing in the definition of the hypergeometric series in our notation. \par
Rodriguez Villegas conjectured the following fourteen supercongruences for hypergeometric series appearing as periods of certain families of rigid Calabi-Yau threefolds.  These supercongruences have recently been proven by Long, Tu, Yui, and Zudilin \cite{LTYZ}.
\begin{theorem}[Long, Tu, Yui, Zudilin \cite{LTYZ}]\label{thm:LTYZ-supers}~\\
    Let $r_1, r_2 \in \left\{ \frac{1}{2}, \frac{1}{3}, \frac{1}{4}, \frac{1}{6}\right\}$ or $(r_1, r_2) \in \left\{ \left( \frac{1}{5}, \frac{2}{5}\right), \left( \frac{1}{8}, \frac{3}{8} \right), \left(\frac{1}{10}, \frac{3}{10} \right), \left( \frac{1}{12}, \frac{5}{12}\right) \right\}$.  Then for each prime $p > 5$, the congruence
    \[
        {}_4F_3 \left[ \begin{matrix} r_1 & 1-r_1 & r_2 & 1-r_2 \\ & 1 & 1 & 1 \end{matrix} \hspace{1ex}; \hspace{1ex} 1 \right]_{p-1} \equiv a_p(f_{\left\{r_1, 1-r_1, r_2, 1-r_2\right\}}) \pmod{p^3}
    \]
    holds, for some explicit modular form $f_{\left\{r_1, 1-r_1, r_2, 1-r_2\right\}}$ of weight 4.
\end{theorem}
Proofs of individual cases of Theorem \ref{thm:LTYZ-supers} had previously been given by Kilbourn \cite{Kilbourn},  McCarthy \cite{McCarthy}, and Fuselier and McCarthy \cite{FM}.  Long, Tu, Yui, and Zudilin \cite{LTYZ} give the first unified proof of all cases of this theorem, and in fact give two proofs.  The first uses $p$-adic techniques and holds for all ordinary primes --- primes at which $a_p(f) \not\equiv 0 \pmod{p}$.  The second proof uses character sum arguments and establishes the supercongruence for all primes greater than $7$.
Recently, Long \cite{Long20} has conjectured further supercongruences for hypergeometric functions following a similar shape to those in Theorem \ref{thm:LTYZ-supers}, but with more complicated parameters appearing in $\mathbf{\beta}$.
\begin{conjecture}[Long \cite{Long20}]\label{conj:long-conjectures}~\\[2mm]
    Let $(r_1, r_2, q) \in \left\{\left( \frac{1}{2}, \frac{1}{2}, \frac{4}{3} \right), \left(\frac{1}{2}, \frac{1}{2}, \frac{7}{6}\right), \left(\frac{1}{2}, \frac{1}{3}, \frac{7}{6}\right), \left(\frac{1}{2}, \frac{1}{3}, \frac{5}{4}\right),  \left( \frac{1}{2}, \frac{1}{4}, \frac{7}{6} \right),  \left( \frac{1}{2}, \frac{1}{2}, \frac{5}{4} \right)\right\}$.  Set $\mathbf{\alpha}_{(r_1, r_2)} \colonequals \left\{r_1, 1-r_1, r_2, 1-r_2\right\}$ and $\mathbf{\beta}_q \colonequals \left\{1, 1, q, 2-q\right\}$.  Let $\mathrm{HD}_{(r_1, r_2, q)}$ denote the hypergeometric datum $(\mathbf{\alpha}_{(r_1, r_2)}, \mathbf{\beta}_{q}, 1 )$.  For each of these hypergeometric data, there exists an explicit weight 4 modular form $f_{(r_1, r_2, q)}$ and Dirichlet character $\chi_{(r_1, r_2, q)}$ such that, for all primes $p \geq 7$,
    \[
        p \cdot {}_4F_3 \left[ \begin{matrix} r_1 & 1-r_1 & r_2 & 1-r_2 \\ & 1 & q & 2-q \end{matrix} \hspace{1ex} ; \hspace{1ex} 1 \right]_{p-1} \equiv \chi_{(r_1,r_2,q)}(p) a_p(f_{\left\{r_1, r_2, q\right\}}) \pmod{p^3}.
    \]
\end{conjecture}
Long, Tu, Yui, and Zudilin's \cite{LTYZ} $p$-adic approach consists of two major components.  First, they use $p$-adic methods of Dwork \cite{Dwork} and known modularity of the particular Calabi-Yau threefolds from which these hypergeometric series arise to reduce Theorem \ref{thm:LTYZ-supers} to the following congruence, which involves only the hypergeometric series.
\begin{theorem}[Long--Tu--Yui--Zudilin \cite{LTYZ}]\label{thm:LTYZ-reduced-supers}
    Let $\alpha = \left\{r_1, 1-r_1, r_2, 1-r_2 \right\}$, and $\beta = \left\{1, 1, 1, 1\right\}$.  For each choice of $(r_1, r_2)$ appearing in Theorem \ref{thm:LTYZ-supers}, we have
    \[
        F_{s+1}(\alpha, \beta) \equiv F_s(\alpha, \beta) F_s(\alpha, \beta) \pmod{p^3}.
    \]
\end{theorem}
They then establish this reduced congruence using $p$-adic techniques.  The main goal of this paper is to generalize the approach taken by Long, Tu, Yui, and Zudilin to prove an analogous version of Theorem \ref{thm:LTYZ-reduced-supers} for the hypergeometric series appearing in Conjecture \ref{conj:long-conjectures}.
\begin{theorem}\label{thm:dwork-cong}
    Fix a prime $p \geq 7$.  For each of the six hypergeometric data $\mathrm{HD}_{(r_1, r_2, q)}$ appearing in Conjecture \ref{conj:long-conjectures} and all $s \geq 0$ we have
    \begin{equation*}
        p^{s+1} F_{s+1}(\mathbf{\alpha}, \mathbf{\beta}) \equiv p^s F_s(\mathbf{\alpha}, \mathbf{\beta}) \cdot pF_1(\mathbf{\alpha}, \mathbf{\beta}) \pmod{p^3}.
    \end{equation*}
\end{theorem}
To fully prove Conjecture \ref{conj:long-conjectures} from Theorem \ref{thm:dwork-cong} it would be necessary to generalize Dwork's work for these cases and to establish the expected modularity of the Galois representations associated to our hypergeometric data by Katz \cite{Katz90, Katz09} and Beukers, Cohen, and Mellit \cite{BCM}.  Li, Long, and Tu \cite{LLT} have recently proven this modularity for the hypergeometric data corresponding to the first three tuples $(r_1, r_2, q)$ listed in Conjecture \ref{conj:long-conjectures}.  The Dirichlet characters and modular forms corresponding to these hypergeometric data are listed below in Figure \ref{fig:modular-forms-for-hypergeometric-data}. \par
\begin{figure}[ht]\label{fig:modular-forms-for-hypergeometric-data}
    \centering
    \begin{tabular}{|c|c|c|c|}
        \hline 
             $(r_1, r_2, q)$ &  $(c,f)$ & $\chi_{(r_1, r_2, q)}$ & $f_{(r_1, r_2, q)}$ \\
        \hline
        \rule{0pt}{3ex}$\left( \frac{1}{2}, \frac{1}{2}, \frac{4}{3} \right)$ & $\left( \frac{1}{6}, \frac{1}{2}\right)$ & $\epsilon$ & $f_{24.4.a.a}$ \\[2mm]
        \hline
        \rule{0pt}{3ex}$\left( \frac{1}{2}, \frac{1}{2}, \frac{7}{6} \right) $ & $\left(\frac{1}{3}, \frac{1}{2}\right)$ & $\epsilon$ &    $f_{12.4.a.a}$ \\[2mm]
        \hline
        \rule{0pt}{3ex}$\left( \frac{1}{2}, \frac{1}{3}, \frac{7}{6} \right)$ & $\left(\frac{1}{3} , \frac{1}{3}\right)$ & $\left(\frac{3}{\cdot}\right)$ & $f_{48.4.a.c}$ \\[2mm]
        \hline
    \end{tabular}
    \caption{The values of $\chi$ and $f$ in Conjecture \ref{conj:long-conjectures} for the hypergeometric data considered by Li, Long, and Tu.  The modular forms are listed by their LMFDB \cite{LMFDB} labels}
\end{figure}
Watkins' \cite{Watkins} implementation of hypergeometric motives in \texttt{Magma} can be used to check a generalization of Dwork's results to these hypergeometric series.  The following conjecture holds computationally for all primes up to 1000 when $s=1$, all primes up to 100 when $s = 2$, and all primes up to 50 when $s = 3$ or $4$.
\begin{conjecture}\label{conj:dwork-non-integral}
    Assume $p \geq 7$ is a prime, and let $s \geq 1$.  For each $\mathrm{HD}_{(r_1, r_2, q)}$ we are considering, $pF_s(\alpha,\beta)/F_{s-1}(\alpha, \beta) \in \ZZ_p$ and there exists $\gamma_p = \gamma_p(\alpha, \beta) \in \ZZ_p^\times$ such that
    \[
        p \frac{F_s(\alpha, \beta)}{F_{s-1}(\alpha, \beta)} \equiv \gamma_p \mod{p^s}.
    \]
    Moreover, letting $\rho_{(\alpha, \beta)}$ denote the Galois representation of the absolute Galois group of $G_\QQ$ associated to the hypergeometric data $(\alpha, \beta, 1)$ by Beukers, Cohen, and Mellit \cite{BCM}, the limit $\gamma_p$ is a unit root of the characteristic polynomial of $\rho_{(\alpha, \beta)}(\mathrm{Frob}_p)$.
\end{conjecture}
A proof of Conjecture \ref{conj:dwork-non-integral}, along with the modularity of the given Galois representations, would together imply Conjecture \ref{conj:long-conjectures} from Theorem \ref{thm:dwork-cong}.  However, the fact that our hypergeometric series do not always have $p$-adically integral coefficients means that this conjecture does not generalize from the approach taken by Dwork.  \par
The rest of the paper is organized as follows.  First, we record some necessary background about $p$-adic interpretations of hypergeometric functions in Section \ref{sec:background}.  In Section \ref{sec:valuations} we prove some preliminary results on the $p$-adic valuations of the given hypergeometric coefficients.  Section \ref{sec:proofs} is dedicated to the proof of Theorem \ref{thm:dwork-cong}.

\subsection{Acknowledgments}
The author thanks Ling Long for many helpful conversations throughout the course of this work.  Additionally, the author thanks Robert Lemke Oliver for assistance in computing data in support of Conjecture \ref{conj:dwork-non-integral}.

\section{$p$-adic Background}\label{sec:background}
Throughout the paper we work over the $p$-adic integers $\ZZ_p$.  The first of the two proofs of Theorem \ref{thm:LTYZ-supers} appearing in \cite{LTYZ} utilizes $p$-adic perturbation techniques originating in \cite{CLZ, Long11} and further developed in \cite{LR}.  As the Pochhammer symbol $(a_k)$ can be written as $\Gamma(x+n)/\Gamma(x)$, we can translate into the $p$-adic setting using Morita's $p$-adic $\Gamma$-function \cite{Morita}, which is defined on integers $n$ by setting
\begin{equation}\label{eq:gamma-p-defn} \notag
    \Gamma_p(x) = (-1)^x \prod_{\substack{0 < i < x \\ p \nmid i}} i,
\end{equation}
and then extended continuously to the $p$-adic integers $\ZZ_p$.  By comparing (\ref{eq:gamma-p-defn}) to the definition of the classical $\Gamma$ function we obtain the following identity for all integers $n$:
\begin{equation}\label{eq:gamma-to-p-gamma}
    \Gamma(n) = (-1)^n \Gamma_p(n) \biggr\lfloor \frac{n-1}{p} \biggr\rfloor ! p^{\lfloor \sfrac{(n-1)}{p} \rfloor}.
\end{equation}
For an overview of the function $\Gamma_p$ and its applications see for example Diamaond \cite{Diamond}.  The function $\Gamma_p$ satisfies the following identities, which are analogous to the functional equation and Euler's reflection formula for the classical $\Gamma$ function.
\begin{lemma}\label{lem:p-gamma-properties}~\\
    Let $x \in \ZZ_p$ and $\Gamma_p$ be defined as in (\ref{eq:gamma-p-defn}).  We have
    \begin{equation*}\label{eq: p-gamma-functional-equation}
        \frac{\Gamma_p(x+1)}{\Gamma_p(x)} = 
        \begin{cases}
            -x, & \text{if } x \in \ZZ_p^{\times} \\
            -1, & \text{if } x \in p\ZZ_p. 
        \end{cases}
    \end{equation*}
        and
    \begin{equation*}\label{eq:p-gamma-reflection-formula}
        \Gamma_p(x)\Gamma_p(1-x) = (-1)^{x_0}
    \end{equation*}
    where $x_0 \in \left\{1, 2, \cdots, p\right\}$ satisfies $x-x_0 \equiv 0 \pmod{p}$.  Note that $x_0$ differs slightly from our previous definition of $[x]_0$, as $x_0 = p$ when $[x]_0 = 0$.  In all other cases these quantities agree.
\end{lemma}
For each $k \geq 0$, define the function $G_k(a) \colonequals \Gamma_p^{(k)}(a)/\Gamma_p(a)$.  These functions are considered by Long and Ramakrishna in \cite{LR}, and satisfy many nice analytic properties.  For example, logarithmically differentiating the second identity in Lemma (\ref{eq:p-gamma-reflection-formula}) yields the reflection formula
\begin{equation}\label{eq:G1-reflection}
    G_1(a) = G_1(1-a).
\end{equation}
Rewriting in terms of $\Gamma_p$ and differentiating again produces
\begin{equation}\label{eq:G2-reflection}
    G_2(a)+G_2(1-a)=2G_1^2(a).
\end{equation}
Long and Ramakrishna \cite{LR} additionally show that $G_2(0) = G_1^2(0)$.  Combining this with (\ref{eq:G1-reflection}) and (\ref{eq:G2-reflection}) when $a = 0$ gives us the related identity
\begin{equation}\label{eq:G1-to-G2-at-1}
    G_1^2(1) = G_2(1).
\end{equation}
The following theorem of Long and Ramakrishna makes the $p$-adic approach particularly appealing for our supercongruences as it allows us to rewrite quotients of $\Gamma_p$ functions, which arise naturally from hypergeometric functions, as a $p$-adic series involving the $G$ functions.
\begin{theorem}[Long--Ramakrishna, \cite{LR}]\label{thm:p-gamma-approx}
    For $p \geq 5$, $r \in \NN, a \in \ZZ_p, m \in \CC_p$ satisfying $v_p(m) \geq 0$ and $t \in \left\{0, 1, 2 \right\}$ we have
    \[
        \frac{\Gamma_p(a+mp^r)}{\Gamma_p(a)} \equiv \sum_{k=0}^t \frac{G_k(a)}{k!}(mp^r)^k \mod p^{(t+1)r}.
    \]
    The above result also holds for $t = 4$ if $p \geq 11$.
\end{theorem}
Again we can take logarithmic derivatives to obtain identities for the $G$ functions as well.  The only cases of these related identities we will need are when $t=0$ and $r=1$, in which case we find that, for $k \in \left\{1,2\right\}$,
\begin{equation}\label{eq:Gk-p-shift}
    G_k(a+mp) - G_k(a) = O(p).
\end{equation}
Also of use to us is Dwork's framework for $p$-adic hypergeometric functions \cite{Dwork}.  Throughout the paper we use $\lfloor x \rfloor$ to denote the floor function, $\left\{ x \right\}$ to denote the fractional part $x - \lfloor x \rfloor$, and for any $a \in \ZZ_p$ we use $[a]_0$ to denote the first $p$-adic digit of $a$.  Dwork's dash operation is the map $': \QQ \cap \ZZ_p \to \QQ \cap \ZZ_p$ defined by
\[
    a' = \frac{a+[-a]_0}{p}.
\]
Each of the six possible choices of $\mathbf{\alpha} = \mathbf{\alpha}_{(r_1, r_2)}$ appearing in Theorem \ref{thm:dwork-cong} are closed under this operation for all primes $p \geq 7$.  That is, as multisets,
\[
    \mathbf{\alpha} = \left\{r_1, 1-r_1, r_2, 1-r_2\right\} = \left\{r_1', (1-r_1)', r_2', (1-r_2)' \right\} = \mathbf{\alpha}'.
\]
Our choices of $\mathbf{\beta} = \mathbf{\beta}_{q}$ are not closed under the Dwork dash operation, instead we have 
\begin{equation}\label{eq:beta-dwork-dash}
    \mathbf{\beta}' = \left\{1', 1', q', (2-q)'\right\} = \left\{1, 1, q-1, 2-q\right\}.
\end{equation}
The exact relation between $\beta$ and $\beta'$, in particular whether $q-1 = q'$ or $q-1 = (2-q)'$, depends on the congruence of $p$ modulo the denominator of $q$. \par
For the remainder of the paper, with our multisets $\mathbf{\alpha}$ and $\mathbf{\beta}$ and our prime $p \geq 7$ fixed, we relabel $\mathbf{\alpha} = \left\{r_1, r_2, r_3, r_4\right\}$ and $\mathbf{\beta} = \left\{ q_1, q_2, q_3, q_4 \right\}$ so that 
\begin{equation}\label{eq:dwork-dash-ordering}
    r_1' \leq r_2' \leq r_3' \leq r_4' \quad \text{and} \quad q_1' \leq q_2' \leq q_3' \leq q_4'. 
\end{equation}
For each $\mathrm{HD}_{(r_1, r_2, q)}$, this choice ensures that $r_2 = r_3 = \frac{1}{2}$ and $q_3 = q_4 = 1$.  For each $a \in \QQ \cap \ZZ_p$, we observe
\begin{equation}\label{eq:dwork-dash-difference}
    1-a' = \frac{p-a-[-a]_0}{p} = \frac{1-a+p-1-[-a]_0}{p} = \frac{1-a+[a-1]_0}{p} = (1-a)'.
\end{equation}
By definition of $\alpha$, $1-r_1$ must belong to $\alpha$, so one of $r_2, r_3$ or $r_4$ must equal $1-r_1$.  If $r_i = 1-r_1$, (\ref{eq:dwork-dash-difference}) implies that $r_i' = 1-r_1'$.  Our choice of ordering in (\ref{eq:dwork-dash-ordering}) guarantees that $r_4' = 1-r_1'$ and hence that $r_4 = 1-r_1$.  A similar relationship holds between $r_2$ and $r_3$.  We rewrite this as
\begin{equation}\label{eq:alpha-sums}
    r_1'+r_4' = r_2'+r_3' = 1 \qquad \text{and} \qquad r_1+r_4 = r_2+r_3 = 1.
\end{equation}
By definition of $\mathbf{\beta}$, we have $\left\{q_1', q_2'\right\} = \left\{q_i-1, 2-q_i\right\}$ where $i \in \left\{1, 2 \right\}$ is chosen so that $q_i > 1$.  Thus, 
\begin{equation}\label{eq:beta-sum}
    q_1' + q_2' = 1 \qquad \text{and} \qquad q_1+q_2 = 2.
\end{equation}
For each $1 \leq j \leq 4$ we define
\begin{equation}\label{eq:ts-us-defn}
    t_j \colonequals [-r_j]_0 = pr_j' - r_j \quad \text{and} \quad u_j \colonequals [-q_j]_0 = pq_j' - q_j.
\end{equation}
As $r_1' \leq r_2'$, the definition of the dash operation implies that $r_1 - r_2 \leq t_2 - t_1$.  But $r_1-r_2 > -1$ and $t_2 - t_1 \in \ZZ$, and so $t_1 \leq t_2$.  Additionally, $r_1 + r_4 + t_1 + t_4$ is divisible by $p$ by definition of $t_i$ and (\ref{eq:alpha-sums}), and so $t_1 + t_4$ is congruent to $[-(r_1+r_4)]_0$ modulo $p$.  We also know that $r_1+r_4 = 1$ so $[-(r_1+r_4)]_0 = p-1$.  As $0 \leq t_1, t_4 \leq p-1$ it must therefore be the case that $t_1+t_4 = p-1$.  Similar arguments using (\ref{eq:alpha-sums}) and (\ref{eq:beta-sum}) hold for the remaining terms, and so we have
\begin{equation}\label{eq:t-identities}
    t_1 \leq t_2 \leq t_3 \leq t_4 \qquad \text{and} \qquad t_1+t_4 = t_2+t_3 = p-1
\end{equation}
and also
\begin{equation}\label{eq:u-identities}
    u_1 \leq u_2 \leq u_3 \leq u_4 \qquad \text{and} \qquad u_1+u_2 = p-2, \quad u_3 = u_4 = p-1.
\end{equation}
For each $\mathrm{HD}_{(r_1, r_2, q)}$ we have $q_1' < r_i' < q_2'$ for each $i \in \left\{1, 2, 3, 4 \right\}$, and so weaving this together we have
\begin{equation}\label{eq:u-t-comparison}
    u_1 < t_1 \leq t_2 = t_3 \leq t_4 < u_2 < u_3 = u_4 = p-1.
\end{equation}
Finally we observe that with this change of labeling $\beta$ can be written as
\begin{equation}\label{eq:beta-dwork-dash-relabel}
    \mathbf{\beta} = \left\{1, 1, q_1'+1, q_2' \right\}.
\end{equation}

\section{$p$-adic valuations of hypergeometric coefficients.}\label{sec:valuations}
Recall that the $p$-adic valuation of a rational number $r$, which we denote by $v_p(r)$ is equal to the exponent $k$ on $p$ when $r$ is written in the form $(a/b)p^k$ with $a$ and $b$ both relatively prime to $p$.  As we wish to reduce our hypergeometric series modulo $p^3$, it will be useful to know the exact $p$-adic valuations of each of the hypergeometric coefficients defined by
\begin{equation}\label{eq:hypergeometric-coefficient-defn} \notag
    H(k) \colonequals H_{\mathbf{\alpha}, \mathbf{\beta}}(k) \colonequals \frac{(r_1)_k(r_2)_k(r_3)_k(r_4)_k}{(q_1)_k(q_2)_k(1)_k(1)_k},
\end{equation}
where $\alpha, \beta$ are as in Conjecture \ref{conj:long-conjectures}.  Given $a = \sum_{n \geq 0} a_np^n \in \ZZ_p$, set
\begin{equation}\label{eq:p-adic-trunc-defn}
    [a]_i = \sum_{n = 0}^{i} a_n p^n.
\end{equation}
We first consider the valuations of the rising factorials.
\begin{lemma}\label{lemma:rising-factorial-valuation}
    Let $p$ be a prime, let $k \in \NN$, and let $a \in \QQ \cap \ZZ_p$.  For each integer $i \geq 0$, let $[a]_i$ be defined as in (\ref{eq:p-adic-trunc-defn}).  The $p$-adic valuation of $(a)_k$ is given by the formula
    \[
        v_p((a)_k) = \sum_{i=1}^\infty \biggr\lfloor \frac{k+p^i - [-a]_{i-1}-1}{p^i} \biggr\rfloor.
    \]
\end{lemma}
\begin{proof}
Expanding the rising factorial we have
\begin{equation} \notag
    v_p \left(a\right)_k = v_p[a(a+1)\cdots(a+k-1)] = \sum_{j=0}^{k-1} v_p(a+j)).
\end{equation}
For each $i$, the smallest non-negative integer $j$ such that $a+j$ is divisible by $p^i$ is $j = [-a]_{i-1}$.  As the indexing of $j$ begins at zero, this occurs at the $([-a]_{i-1}+1)^{st}$ term of the rising factorial.  Multiples of $p^i$ will also appear at each term which differs from $a+[-a]_{i-1}$ by a multiple of $p^i$.  There are $(k-[-a]_{i-1}-1)/p^i$ such terms after the first which are divisible by $p^i$.  Thus, the total number of terms which are divisible by $p^i$ is equal to
\[
    1 + \biggr\lfloor \frac{k-[-a]_{i-1}-1}{p^i} \biggr\rfloor = \biggr\lfloor \frac{k+p^i-[-a]_{i-1}-1}{p^i} \biggr\rfloor.
\]
So, the sum
\[
    \sum_{i=1}^\infty \biggr\lfloor \frac{k+p^i - [-a]_{i-1} - 1}{p^i} \biggr\rfloor
\]
counts every term in $(a)_k$ which is divisible by $p$ once, every term that is divisible by $p^2$ twice, and so on.  Therefore it is equal to $v_p((a)_k)$, as was to be shown.
\end{proof}
\begin{remark}
    In the case where $a=1$, we have $(1)_k = k!$ and Lemma \ref{lemma:rising-factorial-valuation} reduces to Legendre's formula
    \begin{equation}\notag
        v_p (k!) = \sum_{i=1}^{\infty} \biggr\lfloor \frac{k}{p^i} \biggr\rfloor.
    \end{equation}
\end{remark}
The next lemma allows us to compute the truncation $[a]_i$ as defined in (\ref{eq:p-adic-trunc-defn}) for any $a \in \QQ \cap \ZZ_p$.
\begin{lemma}\label{lem:trunc-calc}
    Let $p$ be a prime and $a/b \in \QQ$ be written in reduced terms such that $b \geq 1$ and $p \nmid b$.  Let $\lambda_i$ denote the least positive residue of $ap^{-i}$ modulo $b$.  Then for each $i \in \NN$,
    \[
        [-a/b]_{i-1} = \frac{\lambda_ip^i-a}{b}.
    \]
\end{lemma}
\begin{proof}
    By definition, $[-a/b]_{i-1}$ is the smallest positive integer such that
    \begin{equation}\label{eq:trunc-div}
        p^i \mid \left(a + b\left[ -a/b \right]_{i-1}\right).
    \end{equation}
    Thus, there exists $n \in \NN$ such that $a + b[-a/b]_{i-1} = np^i$.  Solving this equation for $[-a/b]_{i-1}$ yields
    \begin{equation}\label{eq:trunc-rewrite}
        \left[ -a/b \right]_{i-1} = \frac{np^i -a}{b}
    \end{equation}
    The fact that $[-a/b]_{i-1}$ is the least integer satisfying (\ref{eq:trunc-div}) implies that it is the smallest integer of the form $(np^i-a)/b$. Thus, finding $[-a/b]_{i-1}$ is equivalent to finding the smallest $n \in \NN$ such that $(np^i-a)/b \in \ZZ$.  For this to be an integer, we must have $np^i-a \equiv 0 \pmod{b}$.  Thus,
    \[
        n \equiv ap^{-i} \pmod{b}.
    \]
    As $n$ is the least such integer, it follows that $n = \lambda_i$.  This, along with (\ref{eq:trunc-rewrite}), completes the proof.
\end{proof}
For $\mathbf{\alpha} = \left\{r_1, r_2, r_3, r_4\right\}$, define $\mathbf{\alpha}_i$ as the multiset $\left\{[-r_1]_i, [-r_2]_i, [-r_3]_i, [-r_4]_i\right\}$ for each $i \geq 0$.  For $\mathbf{\beta} = \left\{1, 1, q_1, q_2\right\}$ we define $\mathbf{\beta}_i$ similarly.  Using Lemma \ref{lem:trunc-calc} we are able to compute $\mathbf{\alpha}_i$ and $\mathbf{\beta}_i$ for each $\mathrm{HD}_{(r_1, r_2, q)}$ and for each $i$.  These values are recorded in Figure \ref{fig:negative-truncations}.
\begin{figure}[ht]
    \begin{tabular}{|c|c|}
         \hline
         $\mathbf{\alpha}$ & $\mathbf{\alpha}_i$ \\
         \hline
         \rule{0pt}{4ex}$\left\{ \frac{1}{2}, \frac{1}{2}, \frac{1}{2}, \frac{1}{2} \right\}$ & $\left\{ \frac{p^{i+1}-1}{2}, \frac{p^{i+1}-1}{2}, \frac{p^{i+1}-1}{2}, \frac{p^{i+1}-1}{2} \right\}$ \\[3mm]
         \hline
         \rule{0pt}{4ex}$\left\{ \frac{1}{2}, \frac{1}{2}, \frac{1}{3}, \frac{2}{3} \right\}$ & $\left\{ \bigr\lfloor \frac{p^{i+1}}{3} \bigr\rfloor, \frac{p^{i+1}-1}{2}, \frac{p^{i+1}-1}{2}, \bigr\lfloor \frac{2p^{i+1}}{3} \bigr\rfloor \right\}$ \\[3mm]
         \hline
         \rule{0pt}{4ex}$\left\{ \frac{1}{2}, \frac{1}{2}, \frac{1}{4}, \frac{3}{4} \right\}$ & $\left\{ \bigr\lfloor \frac{p^{i+1}}{4} \bigr\rfloor, \frac{p^{i+1}-1}{2}, \frac{p^{i+1}-1}{2}, \bigr\lfloor \frac{3p^{i+1}}{4} \bigr\rfloor \right\}$ \\[3mm]
         \hline
         \hline
         $\mathbf{\beta}$ & $\mathbf{\beta}_i$  \\
         \hline
         \rule{0pt}{4ex}$\left\{1, 1, \frac{4}{3}, \frac{2}{3} \right\}$ & $\left\{ \left[ \frac{p^{i+1}}{3} \right]-1, \left[ \frac{2p^{i+1}}{3} \right] - 1, p^{i+1}-1, p^{i+1}-1 \right\}$ \\[3mm]
         \hline
         \rule{0pt}{4ex}$\left\{1, 1, \frac{5}{4}, \frac{3}{4}\right\}$ & $\left\{ \left[ \frac{p^{i+1}}{4} \right]-1, \left[ \frac{3p^{i+1}}{4} \right] - 1, p^{i+1}-1, p^{i+1}-1 \right\}$ \\[3mm]
         \hline
         \rule{0pt}{4ex}$\left\{1, 1, \frac{7}{6}, \frac{5}{6} \right\}$ & $\left\{ \left[\frac{p^{i+1}}{6} \right] - 1, \left[ \frac{p^{i+1}}{6} \right] -1, p^{i+1}-1, p^{i+1}-1 \right\}$ \\[3mm]
         \hline
    \end{tabular}
    \caption{$p$-adic truncations for the negatives of the parameters appearing in our choices of $\mathbf{\alpha}$ and $\mathbf{\beta}$}
    \label{fig:negative-truncations}
\end{figure}~\par
From these calculations we conclude the following generalization of (\ref{eq:u-t-comparison}).
\begin{corollary}\label{cor:p-adic-valuation-jumps}
    Let $p \geq 7$ be prime and $i \geq 1$.  With notation as above, label $\mathbf{\alpha}_i = \left\{t_j^{(i)}\right\}_{j=1}^4$ and $\mathbf{\beta}_i = \left\{u_j^{(i)}\right\}_{j=1}^4$ such that
    \[
        t_1^{(i)} \leq t_2^{(i)} \leq t_3^{(i)} \leq t_4^{(i)} \quad \text{and} \quad u_1^{(i)} \leq u_2^{(i)} \leq u_3^{(i)} \leq u_4^{(i)}.
    \]
    Then
    \[
        p^i \leq u_1^{(i)} < t_1^{(i)} \leq t_2^{(i)} = t_3^{(i)} \leq t_4^{(i)} < u_2^{(i)} < u_3^{(i)} = u_4^{(i)} = p^{i+1}-1.
    \]
    In particular, for each $0 \leq k \leq p^i-1$ and each of our hypergeometric data, none of the Pochhammer symbols appearing in $H_{\mathbf{\alpha}, \mathbf{\beta}}(k)$ are divisible by $p^{i+1}$.
\end{corollary}
The following proposition inductively finds the valuations of all $H_{\mathbf{\alpha}, \mathbf{\beta}}(k)$ for each $\mathrm{HD}_{(r_1, r_2, q)}$.  In particular it shows that both sides of the congruence in Theorem \ref{thm:dwork-cong} are $p$-adically integral and hence the congruence is well-defined.
\begin{prop}\label{prop:hypergeometric-coefficient-valuations}
    Let $p \geq 7$ be prime and $s \in \NN$.  For each $\mathrm{HD}_{(r_1, r_2, q)}$ appearing in Conjecture \ref{conj:long-conjectures}, we have $p^sH_{\alpha, \beta}(k) \in \ZZ_p$ for all $0 \leq k \leq p^s-1$.  In particular, $p^sF_s(\alpha, \beta) \in \ZZ_p$.
\end{prop}
\begin{proof}
    By the strong triangle inequality of the $p$-adic valuation we have
    \[
        v_p(F_s(\mathbf{\alpha}, \mathbf{\beta})) \geq \min_{0 \leq k \leq p^s-1} \left\{ v_p(H_{\mathbf{\alpha}, \mathbf{\beta}}(k)) \right\}.
    \]
    Therefore it suffices to show that, for all $0 \leq k \leq p^s-1$, we have $v_p(H(k)) \geq -s$. We induct on $s$.  For the base case $s = 1$,
    \[
        v_p(H(k)) = \sum_{i=1}^4 v_p((r_i)_k) - v_p((q_i)_k)
    \]
    and so by (\ref{eq:u-t-comparison}) and Lemma \ref{lemma:rising-factorial-valuation} we have
    \[
         v_p(H(k)) = 
        \begin{cases}
            0 & 0 \leq k \leq u_1 \\
            -1 & u_1 < k \leq t_1 \\
            0 & u_1 < k \leq t_2 \\
            2 & t_3 < k \leq t_4 \\
            3 & t_4 < k \leq u_2 \\
            2 & u_2 < k \leq p-1.
        \end{cases}
    \]
    This completes the base case. \par
    For the inductive step, fix $s \geq 1$ and assume that $v_p(H(k)) \geq -s$ for all $0 \leq k \leq p^s-1$.  Define a function $f: \NN \to \ZZ$ by setting $f(n) = v_p(H(r))$ where $r$ is the remainder of $n$ when dividing by $p^s$.  The function $f$ relates closely to the $p$-adic valuations of $H(k)$ for $k\geq p^s$.  In particular, it only differs from the actual valuations by failing to account fully for multiples of $p^{s+1}$ which appear in the rising factorials.  We correct for this using Corollary \ref{cor:p-adic-valuation-jumps} and Lemma \ref{lemma:rising-factorial-valuation}.  For example, when $k = u_1^{(s)}+1$, we introduce a multiple of $p^s$ to $(q_1)_k$ by Lemma \ref{lemma:rising-factorial-valuation}.  However, $f$ will only decrease by $s-1$ at this index.  So, for $u_1^{(s)} < k \leq t_1^{(s)}$, we have $v_p(H(k)) = f(k)-1$.  Continuing this reasoning we obtain
    \[
        v_p(H(k)) = 
        \begin{cases}
            f(k) & 0 \leq k \leq u_1^{(s)} \\
            f(k) - 1 & u_1^{(s)} < k \leq t_1^{(s)} \\
            f(k) & t_1^{(s)} < k \leq t_2^{(s)} \\
            f(k) + 2 & t_3^{(s)} < k \leq t_4^{(s)} \\
            f(k) + 3 & t_4^{(s)} < k \leq u_2^{(s)} \\
            f(k) + 2 & u_2^{(s)} < k \leq p^{s+1}_1.
        \end{cases}
    \]
    As $f(k) \geq -s$ for all $k$ by the inductive hypothesis, $v_p(F_{s+1}(\mathbf{\alpha}, \mathbf{\beta})) \geq -(s+1)$.
\end{proof}
The valuations of $pH(k)$ established for $0 \leq k \leq p-1$ in the base case will be of particular use, and so we record them separately in the following corollary for future reference.
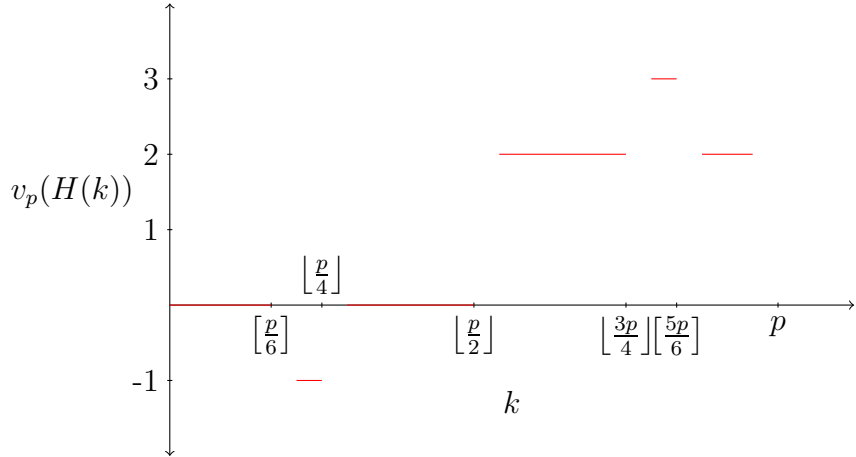
\begin{figure}[ht]\label{fig:hypergeometric-coefficient-valuations}
    \centering
    \begin{tikzpicture}[scale=.3333]
     \draw[->] (0,0)--(27,0);
     \draw[<->] (0,-6)--(0,12);
     \foreach \y in {-1, 1, 2, 3}
     {
        \draw (-.1,3*\y)--(.1,3*\y);
        \node[left] at (0,3*\y) {\y};
     }
     \foreach \x in {4, 6, 12, 18, 20, 24}{
        \draw (\x, -.1)--(\x,.1);
     }
     \node[below] at (4, 0) {$\left[ \frac{p}{6} \right]$};
     \node[above] at (6, 0) {$\left\lfloor \frac{p}{4} \right\rfloor$};
     \node[below] at (12, 0) {$\left\lfloor \frac{p}{2} \right\rfloor$};
     \node[below] at (18, 0) {$\left\lfloor \frac{3p}{4} \right\rfloor$};
     \node[below] at (20,0) {$\left[ \frac{5p}{6} \right]$};
     \node[below] at (24, 0) {$p$};
     \draw[red] (0,0)--(4,0);
     \draw[red] (5,-3)--(6,-3);
     \draw[red] (7,0)--(12,0);
     \draw[red] (13,6)--(18,6);
     \draw[red] (19,9)--(20,9);
     \draw[red] (21,6)--(23,6);
     \node[below] at (13.5, -3) {$k$};
     \node[left] at (-1, 4.5) {$v_p(H(k))$};
    \end{tikzpicture}
    \caption{The $p$-adic valuations of the hypergeometric coefficients $H_{\alpha, \beta}(k)$ corresponding to the hypergeometric datum $(1/2, 1/4, 7/6)$.  By Corollary \ref{cor:hypergeometric-valuations-through-p}, this graph is representative of the shape of hypergeometric valuations for each of our hypergeometric series.}
\end{figure}
\begin{corollary}\label{cor:hypergeometric-valuations-through-p}
    For each $\mathrm{HD}_{(r_1, r_2, q)}$, prime $p \geq 7$, and $0 \leq k \leq p-1$, we have
    \[
        v_p(pH_{\mathbf{\alpha}, \mathbf{\beta}}(k)) = 
        \begin{cases}
            1 & 0 \leq k \leq u_1 \\
            0 & u_1 < k \leq t_1 \\
            1 & u_1 < k \leq t_2 \\
            3 & t_3 < k \leq t_4 \\
            4 & t_4 < k \leq u_2 \\
            3 & u_2 < k \leq p-1.
        \end{cases}
    \]
    In particular, $pH(k) \equiv 0 \pmod{p^3}$ for all $t_3 < k \leq p-1$.
\end{corollary}
%
%
\section{Proof of Theorem \ref{thm:dwork-cong}}\label{sec:proofs}
We now turn our attention to proving Theorem \ref{thm:dwork-cong}, beginning by reinterpreting the supercongruences $p$-adically.  For the remainder of the paper, we fix a prime $p \geq 7$ and $s \in \NN$.  Using the Dwork dash operation, Long, Tu, Yui, and Zudilin \cite{LTYZ} give a formula for rewriting a quotient of shifted factorials in terms of $\Gamma_p(a)$, which we extend from $\ZZ_p^\times$ to all of $\ZZ_p$ in the following Lemma.
\begin{lemma}\label{lem:shift-fact-to-p-gamma}
    Let $k \in \ZZ_{\geq 0}$, $a = [k]_0$, and $b = (k-a)/p$, or $k=a+bp$.  Then for any $r \in \ZZ_p$,
    \[
        \frac{(r)_k}{(1)_k} = \frac{-\Gamma_p(r+k)}{\Gamma_p(1+k)\Gamma_p(r)}\frac{(r')_b}{(1)_b}((r'+b)p)^{\nu(a,[-r]_0)},
    \]
    where
    \[
        \nu(a,x) = -\biggr\lfloor \frac{x-a}{p-1} \biggr\rfloor = \begin{cases} 0 & \text{if } a \leq x, \\ 1 & \text{if } x < a < p. \end{cases}
    \]
\end{lemma}
\begin{proof}
    The case where $r \in \ZZ_p^\times$ is proved by Long, Tu, Yui, and Zudilin \cite{LTYZ}, so we assume $r \in p\ZZ_p$.  Let $h \in \ZZ_p$ be such that $r = ph$.  We note that $[-r]_0 = 0$ and hence $r' = h$.  Then, using (\ref{eq:gamma-to-p-gamma}), we compute
    \[
        \frac{(r)_k}{(1)_k} = \frac{\Gamma(r+k)}{\Gamma(1+k)\Gamma(r)} = \frac{-\Gamma_p(r+k)}{\Gamma_p(1+k)\Gamma_p(r)} \frac{\lfloor (r+k-1)/p \rfloor !}{\lfloor k/p \rfloor ! \lfloor (r-1)/p \rfloor !} p^{\lfloor(r+k-1)/p \rfloor - \lfloor k/p \rfloor - \lfloor (r-1)/p \rfloor}.
    \]
    As $r = ph$ and $k = a + bp$, it follows that $r+k = a+p(b+h)$ and
    \[
        \biggr\lfloor \frac{r+k-1}{p} \biggr\rfloor = h+b + 
        \begin{cases}
            -1 & a = 0 \\
            0 & a > 0.
        \end{cases}
    \]
    Additionally, $\lfloor k/p \rfloor = b$ and $\lfloor (r-1)/p \rfloor = h-1$.  Thus, if $a > [-r]_0 = 0$,
    \begin{align*}
        \frac{(r)_k}{(1)_k}\left(\frac{-\Gamma_p(1+k)\Gamma_p(r)}{\Gamma_p(r+k)}\right) &= \frac{(h+b)!}{b!(h-1)!}p^{(h+b)-b-(h-1)} \\
        &= \frac{(h+b)(h)_b}{(1)_b}p \\
        &= \frac{(r')_b}{(1)_b}(r'+b)p.
    \end{align*}
    And if $a = [r]_0 = 0$,
    \[
        \frac{(r)_k}{(1)_k}\left(\frac{-\Gamma_p(1+k)\Gamma_p(r)}{\Gamma_p(r+k)}\right) = \frac{(h+b-1)!}{b!(h-1)!}p^{(h+b-1)-b-(h-1)} = \frac{(r')_b}{(1)_b}.
    \]
    In each case,
    \[
        \frac{(r)_k}{(1)_k} = \frac{-\Gamma_p(r+k)}{\Gamma_p(1+k)\Gamma_p(r)} \frac{(r')_b}{(1)_b} \left((r'+b)p\right)^{\nu(a,[-r]_0)},
    \]
    as was to be shown.
\end{proof}
\noindent By applying Lemma \ref{lem:shift-fact-to-p-gamma} twice to
\[
    \frac{(r)_k}{(q)_k} = \frac{(r)_k}{(1)_k} \frac{(1)_k}{(q)_k},
\]
we obtain the following.
\begin{corollary}\label{cor:shift-quotient}
    Let $k, a, b$, and $\nu$ be defined as in Lemma \ref{lem:shift-fact-to-p-gamma}.  Then for any $r,q \in \ZZ_p$,
    \[
        \frac{(r)_k}{(q)_k} = \frac{\Gamma_p(r+k)\Gamma_p(q)}{\Gamma_p(q+k)\Gamma_p(r)} \frac{(r')_b}{(q')_b} \frac{(r'+b)^{\nu(a,[-r]_0)}}{(q'+b)^{\nu(a,[-q]_0)}} \cdot p^{\nu(a,[-r]_0)-\nu(a,[-q]_0)}.
    \]
\end{corollary}
 In the particular case where $b=0$ and thus $k=a$, we have
\[
    \frac{\Gamma_p(r+a)\Gamma_p(q)}{\Gamma_p(q+a)\Gamma_p(r)} = \frac{(r)_a}{(q)_a} \frac{(q')^{\nu(a,[-q]_0)}}{(r')^{\nu(a,[-r]_0)}} \cdot p^{\nu(a,[-q]_0)-\nu(a,[-r]_0)}.
\]
And so for general $k=a+bp$,
\begin{singnumalign}
    \frac{(r)_{a+bp}}{(q)_{a+bp}} &= \frac{\Gamma_p(r+a)\Gamma_p(q)}{\Gamma_p(q+a)\Gamma_p(r)} \frac{(r')_b}{(q')_b} \frac{(r'+b)^{\nu(a,[-r]_0)}} {(q'+b)^{\nu(a,[-q]_0)}} \\
    & \hspace{10mm} \times p^{\nu(a,[-r]_0)-\nu(a,[-q]_0)} \frac{\Gamma_p(q+a)\Gamma_p(r+k)}{\Gamma_p(r+a)\Gamma_p(q+k)} \\
    &= \frac{(r)_a}{(q)_a}\frac{(r')_b}{(q')_b} \left(1+\frac{b}{r'}\right)^{\nu(a,[-r]_0)} \left( 1+\frac{b}{q'}\right)^{-\nu(a,[-q]_0)} \\
    &\hspace{30mm} \times \frac{\Gamma_p(q+a)\Gamma_p((r+a)+bp)}{\Gamma_p(r+a)\Gamma_p((q+a)+bp)}. \label{eq:shift-fact-reduction}
\end{singnumalign}
We collect terms in (\ref{eq:shift-fact-reduction}) by defining
\[
    \Lambda_{\mathbf{\alpha},\mathbf{\beta}}(a+bp) \colonequals \prod_{j=1}^4 \left(1+\frac{b}{r_j'}\right)^{\nu(a,t_j)}\left(1+\frac{b}{q_j'}\right)^{-\nu(a,u_j)},
\]
where $t_j$ and $u_j$ are defined as in (\ref{eq:ts-us-defn}).  Using Theorem \ref{thm:p-gamma-approx}, we compute
\begin{singnumalign}\label{eq:r-product}
    \prod_{j=1}^4 \Gamma_p((r_j+a)+bp) &\equiv \prod_{j=1}^4 \Gamma_p(r_j+a) \Biggr( 1+bp\left(\sum_{i=1}^4G_1(r_i+a)\right) \\
    + (bp)^2 \biggr[\frac{1}{2} \sum_{i=1}^4 &G_2(r_i+a)+
    \!\!\!\! \sum_{1\leq i < j \leq 4}\!\!\!\! G_1(r_i+a)G_1(r_j+a)\biggr] \Biggr) \pmod{p^3}.
\end{singnumalign}
Similarly,
\begin{singnumalign}\label{eq:q-product}
    \prod_{j=1}^4\Gamma_p((q_j+a)+bp) &\equiv \prod_{j=1}^4 \Gamma_p(q_j+a) \Biggr( 1+bp\left(\sum_{i=1}^4 G_1(q_i+a)\right) \\
    +(bp)^2\biggr[ \frac{1}{2}\sum_{i=1}^4 &G_2(q_i+a) + \!\!\!\! \sum_{1 \leq i < j \leq 4}\!\!\!\! G_1(q_i+a)G_1(q_j+a) \biggr]\Biggr) \pmod{p^3}.
\end{singnumalign}
Taking the quotient of (\ref{eq:r-product}) and (\ref{eq:q-product}) we have
\begin{equation}\label{eq:p-adic-hypergeometric-coefficient}
    \prod_{j=1}^{4} \frac{\Gamma_p((r_j+a)+bp)}{\Gamma_p((q_j+a)+bp)} \equiv \prod_{j=1}^{4} \frac{\Gamma_p(r_j+a)}{\Gamma_p(q_j+a)}(1+J_1(a) bp+J_2(a)(bp)^2) \pmod{p^3},
\end{equation}
where $J_1 : \ZZ_p \to \ZZ_p$ and $J_2: \ZZ_p \to \ZZ_p$ are defined by
\[
    J_1(a) \colonequals J_1(a;\mathbf{\alpha},\mathbf{\beta}) \colonequals \sum_{j=1}^{4} \left(G_1(r_j+a)-G_1(q_j+a)\right)
\]
and
\begin{align*}
    J_2(a) \colonequals J_2(a;\mathbf{\mathbf{\alpha}},\mathbf{\mathbf{\beta}}) \colonequals& \frac{1}{2}\sum_{j=1}^{4} \left[G_2(r_j+a)-G_2(q_j+a)\right]  \\
    &+ \sum_{i=1}^4 G_1(q_i+a) \sum_{j=1}^{4} \left[G_1(q_j+a)-G_1(r_j+a)\right] \\
    &+ \!\!\!\!\sum_{1\leq i<j \leq {4}}\!\!\!\! \left[ G_1(r_i+a)G_1(r_j+a)-G_1(q_i+a)G_1(q_j+a)\right].
\end{align*}
When $\mathbf{\beta} = \left\{1, 1, 1, 1\right\}$, these definitions reduce to the corresponding functions $J_1(a; \alpha)$ and $J_2(a; \alpha)$ considered by Long, Tu, Yui, and Zudilin \cite{LTYZ}. \par
Applying Corollary \ref{cor:shift-quotient} and (\ref{eq:p-adic-hypergeometric-coefficient}) to $p^{s+1}F_{s+1}(\mathbf{\alpha}, \mathbf{\beta})$, we find
\begin{singnumalign}
    p^{s+1}F_{s+1}(\mathbf{\alpha},\mathbf{\beta}) &= p^{s+1}\sum_{a=0}^{p-1}\sum_{b=0}^{p^s-1} \prod_{j=1}^{4} \frac{(r_j)_{a+bp}}{(q_j)_{a+bp}} \\
    &= p^{s+1}\sum_{a=0}^{p-1} \sum_{b=0}^{p^s-1} \prod_{j=1}^{4} \biggr[\frac{(r_j)_a}{(q_j)_a}\frac{(r_j')_b}{(q_j')_b} \left(1+\frac{b}{r_j'}\right)^{\nu(a,t_j)}  \\
    &\hspace{15mm}\times \left(1+\frac{b}{q_j'}\right)^{-\nu(a,u_j)} \frac{\Gamma_p(q_j+a)\Gamma_p((r_j+a)+bp)}{\Gamma_p(r_j+a)\Gamma_p((q_j+a)+bp)}\biggr] \\
    &\equiv p^{s+1}\sum_{b=0}^{p^s-1}\prod_{j=1}^{4} \frac{(r_j')_b}{(q_j')_b} \sum_{a=0}^{p-1} \prod_{j=1}^{4} \frac{(r_j)_a}{(q_j)_a} \\
    & \hspace{4mm} \times \Lambda_{\mathbf{\alpha},\mathbf{\beta}}(a+bp) \left(1+J_1(a)bp+J_2(a)(bp)^2\right) \pmod{p^3}.
\end{singnumalign}
Recall that, for each of our hypergeometric data $\mathrm{HD}_{(r_1, r_2, q)}$, $\mathbf{\alpha}$ is closed under the Dwork dash operation, so $\prod_{j=1}^4 (r_j')_b = \prod_{j=1}^4 (r_j)_b$.
However, $\mathbf{\beta}$ is not closed under the dash operation.  Instead, by (\ref{eq:beta-dwork-dash-relabel})
\begin{singnumalign}\label{eq:beta-dwork-dash-full}
    \prod_{j=1}^4 (q_j')_b = (1)_b(1)_b(q_1')_b(q_2')_b &= (1)_b(1)_b\left[(q_1')(q_1'+1)_b/(q_1'+b)\right](q_2')_b \\
    &= \frac{1}{\sfrac{b}{q_1'}+1} (1)_b(1)_b(q_1)_b(q_2)_b = \frac{1}{\sfrac{b}{q_1'}+1}\prod_{j=1}^4 (q_j)_b.
\end{singnumalign}
Hence,
\begin{singnumalign}\label{eq:hyper-after-dwork-dash}
    p^{s+1}F_{s+1}(\mathbf{\alpha},\mathbf{\beta}) &\equiv p^{s+1} \sum_{b=0}^{p^s-1} \prod_{j=1}^4 \left(\frac{b}{q_1'}+1\right) \frac{(r_j)_b}{(q_j)_b}  \sum_{a=0}^{p-1} \frac{(r_j)_a}{(q_j)_a} \\
    &\hspace{5mm} \times \left[\Lambda_{\mathbf{\alpha},\mathbf{\beta}}(a+bp)(1+J_1(a)bp+J_2(a)(bp)^2 )\right] \pmod{p^3}.
\end{singnumalign}
Pulling the $\sfrac{b}{q_1'}+1$ term inside the inner sum along with one multiple of $p$, we have
\begin{singnumalign}\label{eq:hyper-full-redux}
    p^{s+1}F_{s+1}(\mathbf{\alpha},\mathbf{\beta}) &\equiv p^{s} \sum_{b=0}^{p^s-1} \prod_{j=1}^4 \frac{(r_j)_b}{(q_j)_b}  \sum_{a=0}^{p-1} p\prod_{j=1}^4 \frac{(r_j)_a}{(q_j)_a} \left(\frac{b}{q_1'}+1\right) \\
    &\hspace{5mm} \times\left[\Lambda_{\mathbf{\alpha},\mathbf{\beta}}(a+bp)(1+J_1(a)bp+J_2(a)(bp)^2 )\right] \pmod{p^3}.
\end{singnumalign}
Therefore, Theorem \ref{thm:dwork-cong} reduces to the following congruence
\begin{singnumalign}\label{eq:dwork-reduction}
    \sum_{b=0}^{p^s-1} p^s \prod_{j=1}^4 \frac{(r_j)_b}{(q_j)_b} \sum_{a=0}^{p-1} p\prod_{j=1}^4 &\frac{(r_j)_a}{(q_j)_a}\biggr[ \left(\frac{b}{q_1'}+1\right)\Lambda_{\mathbf{\alpha},\mathbf{\beta}} (a+bp) \\
    &\times \left(1 + J_1(a)bp + J_2(a) (bp)^2 \right) - 1 \biggr] \equiv 0 \pmod{p^3}.
\end{singnumalign}
In fact, we prove the stronger condition that each term of the sum indexed by $b$ in (\ref{eq:dwork-reduction}) vanishes modulo $p^3$.  Define the inner sum $I(b)$ at $b \in \left\{0, 1, \hdots, p^s-1\right\}$ by
\begin{equation}\label{eq:inn-sum-def}
    I(b) \colonequals \sum_{a=0}^{p-1} p\prod_{j=1}^4 \frac{(r_j)_a}{(q_j)_a} \left[ \left(\frac{b}{q_1'}+1\right) \Lambda(a+bp)(1+J_1(a)bp+J_2(a)(bp)^2 ) - 1\right].
\end{equation}
With this notation and by (\ref{eq:dwork-reduction}), to prove Theorem \ref{thm:dwork-cong} it suffices to show that, for each $s \geq 0$ and $0 \leq b \leq p^s-1$,
\begin{equation}\label{eq:inner-sum-suficient-congruence}
    p^s \prod_{j=1}^4 \frac{(r_j)_b}{(q_j)_b} I(b) \equiv 0 \pmod{p^3}.
\end{equation}
To prove this, we first define $C_1$ and $C_2$ as
\begin{singnumalign}\label{eq:C-def}
    C_1 =& p \biggr[ \sum_{a=0}^{u_1} \prod_{j=1}^4 \frac{(r_j)_a}{(q_j)_a}\left(pJ_1(a)+\frac{1}{q_1'}\right) + \sum_{a= u_1+1}^{t_1} p \prod_{j=1}^4  \frac{(r_j)_a}{(q_j)_a} J_1(a) \\
    &+ \sum_{a = t_1+1}^{t_2} \prod_{j=1}^4 \frac{(r_j)_a}{(q_j)_a} \left(J_1(a)p + \frac{1}{r_1'}\right) \biggr]\\
    C_2 =& p^2 \biggr[ \sum_{a=0}^{u_1} \prod_{j=1}^4 \frac{(r_j)_a}{(q_j)_a} \frac{J_1(a)}{q_1'} + \sum_{a = u_1+1}^{t_1} p \prod_{j=1}^4 \frac{(r_j)_a}{(q_j)_a} J_2(a) \\
    &+ \sum_{a = t_1+1}^{t_2} \prod_{j=1}^4 \frac{(r_j)_a}{(q_j)_a}  \frac{J_1(a)}{r_1'} \biggr].
\end{singnumalign}
We aim to show that both of these quantities vanish modulo $p^3$.  By way of the following proposition, this will then prove (\ref{eq:inner-sum-suficient-congruence}) and hence Theorem \ref{thm:dwork-cong}.
\begin{prop}\label{prop:b-redux}
    Let $C_1, C_2$ be as in (\ref{eq:C-def}) and $I(b)$ as in (\ref{eq:inn-sum-def}).  For each $b$ with $0 \leq b \leq p^s-1$, if 
    \[
        C_1b + C_2b^2 \equiv 0 \pmod{p^3},
    \]
    then
    \[
        p^s \prod_{j=1}^4 \frac{(r_j)_b}{(q_j)_b} I(b) \equiv 0 \pmod{p^3}.
    \]
\end{prop}
Before proving Proposition \ref{prop:b-redux} we prove a few useful Lemmas concerning the $p$-adic valuations of particular parts of $I(b)$.  In particular we consider cases of $b$, depending on whether or not the terms involving $b$ in $I(b)$ are $p$-adically integral.  We first determine exactly when this $p$-integrality occurs.
\begin{lemma}\label{lem:I-b-p-integrality}
    Let $0 \leq a \leq p-1$, $b \in \NN$, and $p \geq 7$ be a prime. For each $\mathrm{HD}_{(r_1, r_2, q)}$, and with $u_i$ and $t_i$ defined as in (\ref{eq:ts-us-defn}), the following is true.  If either $a \leq u_2$ or $b \not\equiv -q_2' \pmod{p}$, then $\left(1+b/q_1'\right) \Lambda_{\mathbf{\alpha}, \mathbf{\beta}}(a+bp) \in \ZZ_p$.
\end{lemma}
\begin{proof}
    By definition of $\Lambda$, $\mathbf{\alpha}$, and $\mathbf{\beta}$, we have
    \begin{singnumalign}\label{eq:Lambda-explicit}
        \left(1+\frac{b}{q_1'}\right)&\Lambda_{\mathbf{\alpha}, \mathbf{\beta}}(a+bp) \\
        &= 
        \begin{cases}
            1+b/q_1' & 0 \leq a \leq u_1 \\
            1 & u_1 < a \leq t_1 \\
            1+b/r_1' & t_1 < a \leq t_2 \\
            (1+b/r_1')(1+b/r_2')(1+b/r_3') & t_2 < a \leq t_4 \\
            (1+b/r_1')(1+b/r_2')(1+b/r_3')(1+b/r_4') & t_4 < a \leq u_2 \\
            \frac{(1+b/r_1')(1+b/r_2')(1+b/r_3')(1+b/r_4')}{1+b/q_2'} & u_2 < a \leq p-1.
        \end{cases}
    \end{singnumalign}
    By choice of $\mathbf{\alpha}$, $\mathbf{\beta}$, and $p$, each $r_i'$ and $q_1'$ all belong to $\ZZ_p^{\times}$, and so the numerator of $(1+b/q_1')\Lambda_{\mathbf{\alpha}, \mathbf{\beta}}(a+bp)$ belongs to $\ZZ_p$.  Thus, the only way that this term could fail to be $p$-adically integral is if $u_2 < a \leq p-1$ and $p \mid 1+b/q_2'$.  As $q_2' \in \ZZ_p$, this second condition is equivalent to $b \equiv -q_2' \pmod{p}$, completing the proof.
\end{proof}
\begin{remark}
    The converse of the above lemma does not hold.  It can be the case that $u_2 < a \leq p-1$ and $b \equiv -q_2' \pmod{p}$, but $(1+b/q_1')\Lambda_{\mathbf{\alpha}, \mathbf{\beta}}(a+bp)$ is $p$-adically integral if the numerator is more highly divisible by $p$ then $1+b/q_2'$.
\end{remark}
In the next Lemma, we consider the case where $(1+b/q_1')\Lambda(a+bp) \in \ZZ_p$.  We show that in this case $I(b) \equiv C_1b+C_2b^2 \pmod{p^3}$.  In Proposition \ref{prop:hypergeometric-coefficient-valuations} we showed that, for any $0 \leq b \leq p^s-1$,
\[
    p^s \prod_{j=1}^4 \frac{(r_j)_b}{(q_j)_b} \in \ZZ_p.
\]
Thus, establishing this congruence between $I(b)$ and $C_1b+C_2b^2$ implies Proposition \ref{prop:b-redux} for such $b$.
\begin{lemma}\label{lem:b-expand}
    Let $b$ be such that $\left(1+\sfrac{b}{q_1'}\right)\Lambda_{\mathbf{\alpha}, \mathbf{\beta}}(a+bp) \in \ZZ_p$, and let $C_1$ and $C_2$ be as in (\ref{eq:C-def}).  We have
    \[
        I(b) \equiv C_1b + C_2b^2 \pmod{p^3}.
    \]
\end{lemma}
\begin{proof}
    Assume $b \geq 0$ is such that $\left(1+\sfrac{b}{q_1'}\right) \Lambda_{\mathbf{\alpha}, \mathbf{\beta}}(a+bp) \in \ZZ_p$.  Then the entire term
    \[
        \left[ \left(\frac{b}{q_1'}+1\right) \Lambda_{\mathbf{\alpha},\mathbf{\beta}}(a+bp)(1+J_1(a)bp+J_2(a)(bp)^2 ) - 1\right]
    \]
    is $p$-adically integral.  Thus, for all $a$ such that the $p$-adic valuation of $H(k)$ is at least $2$, the corresponding term of the sum in $I(b)$ will vanish mod $p^3$.  By Corollary \ref{cor:hypergeometric-valuations-through-p},
    \begin{singnumalign} \label{eq:I-b-reduction}
        I(b) \equiv &\sum_{a=0}^{t_2} p \prod_{j=1}^4 \frac{(r_j)_a}{(q_j)_a} \\
        &\times\left[ \left(\frac{b}{q_1'}+1\right) \Lambda(a+bp)(1+J_1(a)bp+J_2(a)(bp)^2 ) - 1\right] \pmod{p^3}.
    \end{singnumalign}
    Using the explicit values of $\Lambda(a+bp)$ in (\ref{eq:Lambda-explicit}), we have
    \begin{align*}
        I(b) \equiv &\sum_{a=0}^{u_1} p\prod_{j=1}^4 \frac{(r_j)_a}{(q_j)_a} \left[ \left(p J_1(a) + \frac{1}{q_1'}\right)b + \left(p^2J_2(a) + \frac{pJ_1(a)}{q_1'} \right)b^2 + \frac{p^2J_2(a)}{q_1'}b^3 \right] \\
        &+ \sum_{a=u_1+1}^{t_1} p\prod_{j=1}^4 \frac{(r_j)_a}{(q_j)_a} \left[ pJ_1(a)b + p^2J_2(a)b^2\right] \\
        &+ \sum_{a=t_1+1}^{t_2} p\prod_{j=1}^4 \frac{(r_j)_a}{(q_j)_a} \left[ \left( pJ_1(a) + \frac{1}{r_1'}\right)b + \left(p^2J_2(a) + \frac{pJ_1(a)}{r_1'}\right)b^2 + \frac{p^2J_2(a)}{r_1'}b^3 \right].
    \end{align*}
    Regrouping the above terms by powers of $b$, we have $I(b) \equiv C_1b+C_2b^2 + C_3b^3 \pmod{p^3}$, where
    \begin{singnumalign}\label{Ccoeffs}
        C_1 =& p \cdot \biggr[ \sum_{a=0}^{u_1} \prod_{j=1}^4 \frac{(r_j)_a}{(q_j)_a}\left(pJ_1(a)+\frac{1}{q_1'}\right) + \sum_{a= u_1+1}^{t_1} p \prod_{j=1}^4  \frac{(r_j)_a}{(q_j)_a} J_1(a) \\
        &+ \sum_{a = t_1+1}^{t_2} \prod_{j=1}^4 \frac{(r_j)_a}{(q_j)_a} \left(J_1(a)p + \frac{1}{r_1'}\right) \biggr]\\
        C_2 =& p^2 \cdot \biggr[ \sum_{a=0}^{u_1} \prod_{j=1}^4 \frac{(r_j)_a}{(q_j)_a} \left(\frac{J_1(a)}{q_1'} + pJ_2(a)\right) + \sum_{a = u_1+1}^{t_1} p \prod_{j=1}^4 \frac{(r_j)_a}{(q_j)_a} J_2(a) \\
        &+ \sum_{a = t_1+1}^{t_2} \prod_{j=1}^4 \frac{(r_j)_a}{(q_j)_a} \left( p J_2(a) + \frac{J_1(a)}{r_1'} \right)\biggr], \\
        C_3 =& p^3\cdot\left(\sum_{a=0}^{u_1} \prod_{j=1}^4 \frac{(r_j)_a}{(q_j)_a} \frac{J_2(a)}{q_1'} + \sum_{a = t_1+1}^{t_2} \prod_{j=1}^4 \frac{(r_j)_a}{(q_j)_a} \frac{J_2(a)}{r_1'}\right).
    \end{singnumalign}
    The term inside of the parentheses is $p$-integral in each case, as the only piece which is not is $H_{\alpha,\beta}(a)$ when $u_1 < a \leq t_1$, but this is made up for in each $C_i$ by an extra factor of $p$ appearing in that portion of the sum.  In particular we see that $C_3 \equiv 0 \pmod{p^3}$, and so
    \[
        I(b) \equiv C_1b+C_2b^2 \pmod{p^3},
    \]
    as was to be shown.
\end{proof}
We are now ready to prove Proposition \ref{prop:b-redux} for all $0 \leq b \leq p^{s-1}$.
\begin{proof}[Proof of Proposition \ref{prop:b-redux}]
    Lemmas \ref{lem:I-b-p-integrality} and \ref{lem:b-expand} give the result in the case where $b \not\equiv -q_2' \pmod{p}$.  Thus, we consider the case where $b \equiv -q_2' \pmod{p}$.  Define $\widetilde{\Lambda}_{\mathbf{\alpha}, \mathbf{\beta}}(a+bp)$ as follows:
    \begin{singnumalign}\label{eq:Lambda-prime}
        \widetilde{\Lambda}_{\mathbf{\alpha}, \mathbf{\beta}}(a+bp) &= \Lambda_{\mathbf{\alpha},\mathbf{\beta}}(a+bp)\left(1+\sfrac{b}{q_2'}\right)^{\nu(a,u_2)} \\
        &=
        \begin{cases}
            \Lambda_{\mathbf{\alpha},\mathbf{\beta}}(a+bp) & 0 \leq a < u_2 \\
            (1+\sfrac{b}{q_2'})\Lambda_{\mathbf{\alpha},\mathbf{\beta}}(a+bp) & u_2 \leq a \leq p-1.
        \end{cases}
    \end{singnumalign}
    From the discussion in the proof of Lemma \ref{lem:I-b-p-integrality}, $(1+b/q_1')\widetilde{\Lambda}_{\mathbf{\alpha},\mathbf{\beta}}(a+bp)$ belongs to $\ZZ_p$.  We define $\widetilde{I}(b)$ similarly to $I(b)$, but with $\Lambda$ replaced by $\widetilde{\Lambda}$.  That is,
    \begin{equation}\label{modifiedinnersum}
        \widetilde{I}(b) = \sum_{a=0}^{p-1} p\prod_{j=1}^4 \frac{(r_j)_a}{(q_j)_a} \left[ \left(\frac{b}{q_1'}+1\right) \widetilde{\Lambda}(a+bp)(1+J_1(a)bp+J_2(a)(bp)^2 ) - 1\right].
    \end{equation}
    As $J_1, J_2 \in \ZZ_p^\times$, it follows that the expression inside of the brackets is $p$-integral.  This means that the $a^{th}$ term of the sum $\widetilde{I}(b)$, which we will denote by $\widetilde{I}_a(b)$, satisfies
    \[
        v_p(\widetilde{I}_a(b)) \geq v_p\left(p \prod_{j=1}^4 \frac{(r_j)_a}{(q_j)_a}\right).
    \]
    Thus, Corollary \ref{cor:hypergeometric-valuations-through-p} implies $\widetilde{I}_a(b)$ vanishes modulo $p^3$ whenever $a > t_2$.  Therefore, modulo $p^3$ we may write $\widetilde{I}(b)$ as
    \[
        \sum_{a=0}^{t_2} p \prod_{j=1}^4 \frac{(r_j)_a}{(q_j)_a} \left[ \left(\frac{b}{q_1'}+1\right) \widetilde{\Lambda}(a+bp)(1+J_1(a)bp+J_2(a)(bp)^2 ) - 1\right].
    \]
    But, we know that $\widetilde{\Lambda}_{\mathbf{\alpha},\mathbf{\beta}}(a+bp) = \Lambda_{\mathbf{\alpha},\mathbf{\beta}}(a+bp)$ for all $0 \leq a \leq t_2$.  Thus, $\widetilde{I}(b)$ simplifies further modulo $p^3$ as
    \[
        \sum_{a=0}^{t_2} p \prod_{j=1}^4 \frac{(r_j)_a}{(q_j)_a} \left[ \left(\frac{b}{q_1'}+1\right) \Lambda_{\mathbf{\alpha},\mathbf{\beta}}(a+bp)(1+J_1(a)bp+J_2(a)(bp)^2 ) - 1\right].
    \]
    This is exactly how we rewrote $I(b)$ in (\ref{eq:I-b-reduction}), and so by the same calculations we have that
    \[
        \widetilde{I}(b) \equiv C_1b+C_2b^2 \pmod{p^3}.
    \]
    To finish the proof of the proposition, it now suffices to show for $b \equiv -q_2' \pmod{p}$ that, if $\widetilde{I}(b)$ vanishes modulo $p^3$ then so does $p^s \prod_{j=1}^4 \frac{(r_j)_b}{(q_j)_b} I(b)$.  To do so, we show
    \begin{equation}\label{eq:I-tilde-I-congruence}
        p^s \prod_{j=1}^4 \frac{(r_j)_b}{(q_j)_b} I(b) \equiv p^s \prod_{j=1}^4 \frac{(r_j)_b}{(q_j)_b} \widetilde{I}(b) \pmod p^3.
    \end{equation}
    From the definition of $I(b)$ and $\widetilde{I}(b)$, the difference of the expressions in the above congruence is
    \begin{singnumalign}\label{eq:I-tilde-I-difference}
        p^s \prod_{j=1}^4& \frac{(r_j)_b}{(q_j)_b} \sum_{a=0}^{p-1} p \prod_{j=1}^4 \frac{(r_j)_a}{(q_j)_a} \\
        &\times \left[ \left( \frac{b}{q_1'} + 1 \right) \left(\Lambda(a+bp) - \widetilde{\Lambda}(a+bp)\right) (1+J_1(a)bp+J_2(a)(bp^2)\right]
    \end{singnumalign}
    By (\ref{eq:Lambda-prime}), the term $\Lambda-\widetilde{\Lambda}$ vanishes for all $a \leq u_2$.  Thus, we need only consider indices $a$ satisfying $u_2 < a \leq p-1$.  For such $a$, we have
    \begin{align*}
        \Lambda(a+bp)-\widetilde{\Lambda}(a+bp)
        &= \left(\frac{1}{1+\sfrac{b}{q_2'}}\right)\widetilde{\Lambda}(a+bp) - \widetilde{\Lambda}(a+bp) \\
        &= \widetilde{\Lambda}(a+bp)\left(\frac{-b}{q_2'(1+\sfrac{b}{q_2'})}\right).
    \end{align*}
    So, (\ref{eq:I-tilde-I-difference}) simplifies to
    \[
        p^s \frac{-b(1+\sfrac{b}{q_1'})}{q_2'(1+\sfrac{b}{q_2'})}\prod_{j=1}^4 \frac{(r_j)_b}{(q_j)_b} \sum_{a = t_2+1}^{p-1} p \prod_{j=1}^4 \frac{(r_j)_a}{(q_j)_a} \widetilde{\Lambda}(a+bp)\left(1+J_1(a)bp + J_2(a)b^2p^2\right).
    \]
    If we show that this expression vanishes mod $p^3$, we will be done.  To do this, we first consider the sum
    \[
        \sum_{a = u_2+1}^{p-1} p \prod_{j=1}^4 \frac{(r_j)_a}{(q_j)_a} \widetilde{\Lambda}(a+bp)\left(1+J_1(a)bp + J_2(a)b^2p^2\right).
    \]
    Each of $\widetilde{\Lambda}(a+bp)$ and $(1+J_1(a)bp+J_2(a)b^2p^2)$ are $p$-adically integral.  By Corollary \ref{cor:p-adic-valuation-jumps}, the hypergeometric coefficient has $p$-adic valuation equal to 2.  Hence, for each $a$ in this range,
    \begin{align*}
        v_p&\left(p \prod_{j=1}^4 \frac{(r_j)_a}{(q_j)_a} \widetilde{\Lambda}(a+bp)\left(1+J_1(a)bp + J_2(a)b^2p^2\right) \right) \\
        &\hspace{40mm}= v_p(p)+v_p\left(\prod_{j=1}^4 \frac{(r_j)_a}{(q_j)_a}\right) \\
        &\hspace{45mm}+ v_p\left( \widetilde{\Lambda}(a+bp) (1+J_1(a)bp+J_2(a)b^2p^2)\right) \\
        &\hspace{40mm}\geq 1 + 2 + 0 = 3.
    \end{align*}
    Thus, we need only show that the expression
    \[
         p^s \frac{-b(1+\sfrac{b}{q_1'})}{q_2'(1+\sfrac{b}{q_2'})}\prod_{j=1}^4 \frac{(r_j)_b}{(q_j)_b}
    \]
    is $p$-adically integral.  To do so, we first note that, as $p \nmid b$,
    \[
        v_p \left( \frac{-b(1+\sfrac{b}{q_1'})}{q_2'}\right) = 0,
    \]
    and so we need only consider
    \[
        p^s\left(\frac{1}{1+\sfrac{b}{q_2'}}\right) \prod_{j=1}^4 \frac{(r_j)_b}{(q_j)_b}.
    \]
    Let $i$ be the largest integer such that $b \equiv -q_2' \pmod{p^i}$.  Then we have
    \[
        v_p\left(\frac{1}{1+\sfrac{b}{q_2'}}\right) = -i.
    \]
    The condition $b \equiv -q_2' \mod p^i$ is equivalent to saying that $b = [-q_2']_{i-1} + np^i$.  Additionally, our labelling of $\mathbf{\beta}$ guarantees us that $q_2' \in \mathbf{\beta}$, and so therefore $[-q_2']_{i-1} \in \mathbf{\beta}_{i-1}$, where $\mathbf{\beta}_{i-1}$ is labelled as in Corollary \ref{cor:p-adic-valuation-jumps}.  Thus, there exists $j \in \left\{1, 2\right\}$ such that 
    \[
        b = u_j^{(i-1)} + np^i.
    \]
    From the inductive procedure to find the $p$-adic valuations of $H(k)$ described in Proposition \ref{prop:hypergeometric-coefficient-valuations}, there must be a decrease in the $p$-adic valuations of $H(k)$ that occurs at $u_j^{(i-1)}+np^i +1$, as at this index a multiple of $p^i$ is added to the denominator.  In particular,
    \[
        v_p(H(b)) = v_p(H(b+1)) + i.
    \]
    From Proposition \ref{prop:hypergeometric-coefficient-valuations} we also know that $v_p(H(b+1)) \geq -s$.  Therefore,
    \[
        v_p \left(p^s \left(\frac{1}{1+\sfrac{b}{q_2'}}\right) \prod_{j=1}^4 \frac{(r_j)_b}{(q_j)_b} \right) \geq s - i + (i-s) = 0.
    \]
    We have thus shown that this term is $p$-adically integral, completing the proof.
\end{proof}
The discussion at the beginning of the section and Proposition \ref{prop:b-redux} imply that Theorem \ref{thm:dwork-cong} follows from establishing that $C_1b+C_2b^2 \equiv 0 \pmod{p^3}$ for all $b$.  Before showing this, we introduce a number of necessary identities.  First, from (\ref{eq:gamma-to-p-gamma}) and the definition of the Pochhammer symbol, for all $0 \leq a \leq p-1$ we have
\begin{equation}\label{eq:shift-fact-to-p-gamma}
    (t)_a = (-1)^a \frac{\Gamma_p(t+a)}{\Gamma_p(t)} (t+[-t]_0)^{\nu(a, [-t]_0)}.
\end{equation}
Taking logarithmic derivatives of both sides we obtain
\begin{equation}\label{eq:log-der}
    \frac{ \frac{d}{dt}(t)_a}{(t)_a} = G_1(t+a)-G_1(t) + \frac{\nu(a, [-t]_0)}{t+[-t]_0}.
\end{equation}
And differentiating again we obtain
\begin{align}\label{eq:second-log-der}
    \frac{\frac{d^2}{dt^2} (t)_a}{(t)_a} &= \left(G_1(t+a)-G_1(t) + \frac{\nu(a, [-t]_0)}{t+[-t]_0} \right)^2 \notag \\
    &+ G_2(t+a)-G_2(t)-G_1^2(t+a) + G_1^2(t) - \frac{\nu(a,[-t]_0)}{(t+[-t]_0)^2}.
\end{align}
Now, to show that $C_1 \equiv C_2 \equiv 0 \pmod{p^3}$, we introduce particular rational functions whose residues we relate to the $C_i$.  This is similar to the approach taken by Long, Tu, Yui, and Zudilin, as well as other authors to establish similar hypergeometric identities \cite{Chu, OSZ, LTYZ, Zudilin}.  We define, for $i \in \left\{1, 2 \right\}$
\[
    R_i(t) = \frac{ \prod_{j=1}^4 (-t+1-pr_j')_{t_j} (-t+q_1)_{u_1+1}}{(-t+1-pq_1')_{u_1+1}(t)_p^{(i+1)}}.
\]
By (\ref{eq:t-identities}) and (\ref{eq:u-identities}), the degree of the numerator of $R_i(t)$ is $2p-1+u_1$, whereas the denominator has degree $(i+1)p+1+u_1$.  For both choices of $i$ the degree of the denominator is at least $2$ greater than that of the numerator, so the residue sum theorem implies that the sum of the residues of $R_i$ is equal to zero.  These rational functions have partial fraction decompositions
\[
    R_1(t) = \sum_{k=0}^{p-1} \frac{A^{(1)}_k}{(t+k)^2} + \sum_{k=1}^{p-1} \frac{B^{(1)}_k}{t+k} + \sum_{k=1}^{u_1+1} \frac{D^{(1)}_k}{(-t+k-pq_1')},
\]
and
\[
    R_2(t) = \sum_{k=0}^{p-1} \left(\frac{A^{(2)}_k}{(t+k)^3} + \frac{E^{(2)}}{(t+k)^2} + \frac{B^{(2)}_k}{t+k}\right) + \sum_{k=1}^{u_1+1} \frac{D^{(2)}_k}{-t+k-pq_1'}.
\]
By the residue theorem, it follows that
\[
    0 = \sum_{k=0}^{p-1} B^{(1)}_k + \sum_{k=1}^{u_1+1} D^{(1)}_k.
\]
and
\[
    0 = \sum_{k=0}^{p-1} B^{(2)}_k + \sum_{k=1}^{u_1+1} D^{(2)}_k.
\]
Our goal is to relate these residue sum to the $C_i$ in order to show that $C_i \equiv 0 \pmod{p^3}$ for each $i$.  We begin with a lemma which shows that each $D^{(i)}_k$ is small $p$-adically.
\begin{lemma}\label{lem:Dk-valuation}
    Set notation as above, and let $1 \leq k \leq u_1+1$.  For $i \in \left\{1, 2\right\}$, there exists $\delta^{(i)}_k \in \ZZ_p$ such that $D^{(i)}_k = p^{(3-i)}\delta_k$.
\end{lemma}
\begin{proof}
    The statement of the lemma is equivalent to saying $v_p(D^{(i)}_k) \geq 3-i$ for all $p$ and $k$.  We compute
    \begin{align*}
        D^{(i)}_k &= (-t+k-pq_1')R_i(t) \biggr\vert_{t = k-pq_1'} \\
        &= \frac{\prod_{j=1}^4 (-k+1+p(q_1'-r_j'))_{t_j} (-k+q_1+pq_1')_{u_1+1}}{(1-k)_{k-1}(1)_{u_1+1-k}(k-pq_1')^{(3-i)}_p}.
    \end{align*}
    Using (\ref{eq:shift-fact-to-p-gamma}), we have first that, for each $j \in \left\{1, 2, 3, 4, \right\}$, 
    \[
        (-k+1+p(q_1'-r_j'))_{t_j} = (-1)^{t_j} \frac{ \Gamma_p(1-k+t_j-pr_j' + pq_1')}{ \Gamma_p(1-k+p(q_1'-r_j'))}\left(p(q_1'-r_j')\right)^{\nu(t_j, k-1)}.
    \]
    As $\Gamma_p$ takes $\ZZ_p$ to $\ZZ_p^{\times}$, the $p$-adic valuation of the above expression depends only on the term $(p(q_1'-r_j'))^{\nu(t_j, k-1)}$. From (\ref{eq:u-t-comparison}), $t_j > u_1$ for all $j$.  Therefore, for $1 \leq k \leq u_1+1$, $t_j > k-1$ and so $\nu(t_j, k-1) = 1$.  For each of our hypergeometric data except $\mathrm{HD}_{(1/2,1/4,7/6)}$ we have $q_1' - r_j' \in \ZZ_p^\times$ for all $p \geq 7$, and in this one exceptional case the same holds for all $p \geq 11$.  For each of these primes, $p(q_1'-r_j')$ has $p$-adic valuation exactly one, and so $(-k+1+p(q_1'-r_j'))_{t_j}$ does as well.  In the exceptional case $\mathbf{\alpha} = \left\{1/2, 1/2, 3/4, 5/4\right\}$ and $\mathbf{\beta} = \left\{ 1, 1, 7/6, 5/6\right\}$ and with $p=7$ and $j$ chosen so that $r_j' = 3/4$ we have $q_1'-r_j' = 1/6-3/4 = 7/12$, and so the $p$-adic valuation of $(-k+1+p(q_1'-r_j'))_{t_j}$ is in fact equal to $2$.  In all cases, we conclude
    \[
        v_p\left( \prod_{j=1}^4 (-k+1+p(q_1'-r_j'))_{t_j}\right) \geq 4
    \]
    for all $1 \leq k \leq u_1+1$. \\~\\
    Next we consider $(-k+q_1+pq_1')_{u_1+1}$.  As before, the $p$-adic valuation depends only on the final term in (\ref{eq:shift-fact-to-p-gamma}), and in this case the exponent is $\nu(u_1+1, [k-q_1-pq_1']_0)$, which reduces to $\nu(u_1+1, [k-q_1]_0)$.  By definition, the leading $p$-adic digit of $-q_1$ is $u_1$, and so as long as $k < [-u_1]_0$, we will have $[k-q_1]_0 = k+u_1$.  This is the case, as by (\ref{eq:u-identities}) $[-u_1]_0 = u_2+2$, and $k \leq u_1+1 < u_2 + 2$.  Therefore, 
    \[
    u_1 + 1 \leq [k-q_1]_0 \leq 2u_1 + 1 < p,
    \]
    and so $\nu(u_1+1, [k-q_1]_0) = 0$ for all $1 \leq k \leq u_1+1$.  In particular we see that $v_p(-k+q_1+pq_1')_{u_1+1} = 0$.  \\~\\
    For the terms in the denominator, we observe that $(1-k)_{k-1}$ and $(1)_{u_1+1-k}$ together contain all $u_1$ integers from $1-k$ to $u_1+1-k$ excluding zero.  As $u_1+1 < p$, there is thus no multiple of $p$ appearing in these two terms, as the only such multiple in this range is the zero we have removed.  Thus, $v_p((1-k)_{k-1}(1)_{u_1+1-k}) = 0$.  Finally, the shifted factorial $(k-pq_1')_p$ will necessarily have exactly one multiple of $p$ appearing within it, namely at the term $k+p-k - pq_1' = p(1-q_1')$. As $1-q_1' \in \ZZ_p^{\times}$ for all $p \geq 7$ and $\mathbf{\beta}$ within our consideration, we conclude that $v_p((k-pq_1')_p^{i+1} = i+1$.  \par
    Putting these calculations together, we conclude that
    \begin{align*}
        v_p(D^{(i)}_k) &\geq 4-(i+1) = 3-i,
    \end{align*}
    completing the argument.
\end{proof}
We are now prepared to prove Theorem \ref{thm:dwork-cong}.  
\begin{proof}[Proof of Theorem \ref{thm:dwork-cong}]
We have seen from Proposition \ref{prop:b-redux} and (\ref{eq:inner-sum-suficient-congruence} that it suffices to show that each $C_i \equiv 0 \pmod{p^3}$, which we will do by showing that a certain multiple of each $C_i$ is equivalent modulo $p^3$ to the residue sums of the $R_i$ functions defined above.  We first consider $i=1$.  To compute the residues $B^{(1)}_k$, we first must compute $A^{(1)}_k$.  This computation can be done directly, as
\begin{align*}
    A^{(1)}_k = (t+k)^2R(t) \biggr\vert_{t = -k} &= \frac{ \prod_{j=1}^4 (k+1-pr_j')_{t_j} (k+q_1)_{u_1+1}}{(k+1-pq_1')_{u_1+1} (-k)_k^2 (1)_{p-k-1}^2} \\
    &= \frac{\prod_{j=1}^4 (k+1-pr_j')_{t_j} (k+q_1)_{u_1+1}}{(k+1-pq_1')_{u_1+1} (1)_k^2 (1)_{p-k-1}^2}.
\end{align*}
We will interpolate this expression $p$-adically using (\ref{eq:shift-fact-to-p-gamma}), but first we make a few observations we will need for our reductions.  By the definitions of the function $\nu$ and of $t_j$ we have
\begin{equation}\label{eq:nu-rewrite}
    \nu(t_j, p-k-1) = \nu(k+1, p-t_j) = \nu(k+1, [r_j]_0).
\end{equation}
For the term $(k+q_1)_{u_1+1}$, we can use (\ref{eq:shift-fact-to-p-gamma}) to see that
\[
    (k+q_1)_{u_1+1} = (-1)^{u_1+1} \frac{\Gamma_p(k+1+pq_1')}{\Gamma_p(k+q_1)} (k+q_1+[-k-q_1]_0)^{\nu(u_1+1, [-k-q_1]_0)}.
\]
As $q_1 = -u_1+pq_1'$, we have that $-k-q_1 = u_1-k-pq_1'$ and so
\[
    [-k-q_1]_0 = [u_1-k-pq_1']_0 = \begin{cases} u_1 - k & 0 \leq k \leq u_1 \\ p+u_1-k & u_1 < k \leq p-1. \end{cases}
\]
Additionally, we have
\begin{align*}
    \nu(u_1+1, [-k-q_1]_0) &= \begin{cases} \nu(u_1+1, u_1-k) & 0 \leq k \leq u_1 \\ \nu(u_1+1, p+u_1-k) & u_1 < k \leq p-1 \end{cases} \\
    &= \begin{cases} 1 & 0 \leq k \leq u_1 \\ 0 & u_1 < k \leq p-1 \end{cases} \\
    &= \nu(u_1+1,k).
\end{align*}
Combining these two calculations yields
\[
    (k+q_1+[-k-q_1]_0)^{\nu(u_1+1, [-k-q_1]_0)} = (pq_1')^{\nu(u_1+1,k)}.
\]
We choose $i \in \left\{1, 2\right\}$ such that $u_i \ne 1$, which is possible by (\ref{eq:u-identities}).  This choice allows us to use Lemma \ref{lem:p-gamma-properties} to conclude that 
\[
    \Gamma_p(q_1)\Gamma_p(q_2) = \Gamma_p(q_i)\Gamma_p(1-q_i) \frac{\Gamma_p(2-q_i)}{\Gamma_p(1-q_i)} = (-1)^{p-u_i+1}(1-q_i) = (-1)^{u_i+1+\epsilon_i}q_1',
\]
where $\epsilon_i \in \left\{0, 1\right\}$ depends only on $\mathbf{\beta}$ and $p$.  To finish simplifying $A_k^{(1)}$ we use Lemma \ref{lem:p-gamma-properties}, (\ref{eq:shift-fact-to-p-gamma}), (\ref{eq:shift-fact-to-p-gamma}), (\ref{eq:nu-rewrite}), as well as Theorem \ref{thm:p-gamma-approx} and its logarithmic derivative.  This gives
\begin{align*}
    A^{(1)}_k &= \frac{ \prod_{j=1}^4 \frac{\Gamma_p(k+1-r_j)}{\Gamma_p(k+1-pr_j')}\left(p(1-r_j')\right)^{\nu(k+1, [r_j]_0)} (-1)^{u_1+1} \frac{\Gamma_p(k+1+pq_1')}{\Gamma_p(k+q_1)}(pq_1')^{\nu(u_1+1,k)}}{(-1)^{u_1+1} \frac{\Gamma_p(k+q_2)}{\Gamma_p(k+1-pq_1')}(pq_2')^{\nu(k,u_2)} \Gamma_p^2(k+1) \Gamma_p^2(p-k)} \\
    &= \frac{ \prod_{j=1}^4 \Gamma_p(-r_j)(-r_j)_{k+1} \Gamma_p(k+1-pq_1') \Gamma_p(k+1+pq_1')(pq_1')}{ \prod\limits_{i=1}^2\Gamma_p(q_i)(q_i)_k \frac{\Gamma_p(k+q_2)}{\Gamma_p(k+1-pq_1')}(pq_2')^{\nu(k,u_2)} \Gamma_p^2(k+1) \Gamma_p^2(p-k) \prod\limits_{j=1}^4 \Gamma_p(k+1-pr_j')} \\
    &= (-1)^{t_1+t_2}p \prod_{j=1}^4 \frac{(r_j)_k}{(q_j)_k} \frac{q_1'\Gamma_p(k+1+pq_1')\Gamma_p(k+1-pq_1')}{\Gamma_p(q_1)\Gamma_p(q_2) \Gamma_p^2(p-k) \prod_{j=1}^4 \Gamma_p(k+1-pr_j')} \\
    &= (-1)^{t_1+t_2+u_i+1}p \prod_{j=1}^4 \frac{(r_j)_k}{(q_j)_k} \frac{q_1' \Gamma_p(k+1+pq_1') \Gamma_p^2(k+1-p)\Gamma_p(k+1-pq_1')}{(q_i-1)\prod_{j=1}^4 \Gamma_p(k+1-pr_j')} \\
    &= (-1)^{t_1+t_2+u_i+\epsilon_i+1}p \prod_{j=1}^4 \frac{(r_j)_k}{(q_j)_k} (1+O(p)).
\end{align*}
The residues $B^{(1)}_k$ can be computed explicitly as well.  For each $0 \leq k \leq p-1$, we have
\begin{align*}
    B^{(1)}_k &= \lim_{t \to -k} \frac{d}{dt}\left((t+k)^2R(t)\right) \\
    &= A(k) \biggr( \sum_{j=1}^4 \frac{\frac{d}{dt}(-t+1-pr_j')_{t_j} }{(-t+1-pr_j')_{t_j}} + \frac{ \frac{d}{dt}(-t+q_1)_{u_1+1}}{(-t+q_1)_{u_1+1}} - \frac{ \frac{d}{dt}(-t+1-pq_1')_{u_1+1}}{(-t+1-pq_1')_{u_1+1}} \\
    &\hspace{30mm} - 2\frac{\frac{d}{dt} (t)_k}{(t)_k} - 2\frac{(t+k+1)_{p-k-1}}{(t+k+1)_{p-k-1}} \biggr) \biggr\vert_{t = -k}
\end{align*}
Each of these logarithmic derivatives can be written in terms of the $G$ functions using (\ref{eq:log-der}), and subsequently we use (\ref{eq:Gk-p-shift}) to simplify these expressions.  After a large amount of simplification, one finds that
\[
    B^{(1)}_k = -A^{(1)}_k \left( J_1(k;\mathbf{\alpha}, \mathbf{\beta}) + \sum_{j=1}^4 \frac{\nu(k, t_j)}{pr_j'} + \frac{\nu(u_1+1,k)}{pq_1'} - \frac{\nu(k, u_2)}{pq_2'} \right) + O(p^2).
\]
Therefore, if we consider the full residue sum multiplied by $(-1)^{t_1+t_2+u_i+\epsilon_i}p$ and use the $p$-adic valuations from Corollary \ref{cor:hypergeometric-valuations-through-p} and Lemma \ref{lem:Dk-valuation}, we have
\begin{align*}
    0 =& (-1)^{t_1+t_2+u_i+\epsilon_i}p \left(\sum_{k=0}^{p-1} B^{(1)}_k + \sum_{k=1}^{u_1+1} D^{(1)}_k\right) \\
    =& \sum_{k=0}^{p-1} p^2 \prod_{j=1}^4 \frac{(r_j)_k}{(q_j)_k} \left(J_1(k) +\sum_{j=1}^4 \frac{\nu(k, t_j)}{pr_j'} + \frac{\nu(u_1+1, k)}{pq_1'} - \frac{\nu(k, u_2)}{pq_2'} \right) + O(p^3) \\
    =& \sum_{k=0}^{u_1} p \prod_{j=1}^4 \frac{(r_j)_k}{(q_j)_k} \left( pJ_1(k) + \frac{1}{q_1'} \right) + \sum_{k=u_1+1}^{t_1} p \prod_{j=1}^4 \frac{(r_j)_k}{(q_j)_k} pJ_1(k) \\
    &+ \sum_{k=t_1+1}^{t_2} p \prod_{j=1}^4 \frac{(r_j)_k}{(q_j)_k} \left( pJ_1(k) + \frac{1}{r_1'} \right) + O(p^3) \\
    = &C_1+O(p^3).
\end{align*}
Therefore, $C_1 \equiv 0 \pmod{p^3}$.
Showing that $C_2 \equiv 0 \pmod{p^3}$ is very similar to the proof for $C_1$, although the particular details are significantly more involved as we now must compute residues at poles of order 3.  As such, we again omit many of the details.  To calculate $A^{(2)}_k$, we choose $i \in \left\{1,2\right\}$ so that $u_i \ne 1$ which allows for the simplifications
\begin{align*}
    A^{(2)}_k &= \frac{\prod_{j=1}^4 (k+1-pr_j')_{t_j} (k+q_1)_{u_1+1}}{(k+1-pq_1')_{u_1+1}(-1)^k(1)_k^3(1)_{p-1-k}^3} \\
    &= (-1)^{t_1+t_1+k+1}p \prod_{j=1}^4 \frac{(r_j)_k}{(q_j)_k} \frac{q_1' \Gamma_p(k+1+pq_1') \Gamma_p(k+1-pq_1')}{\Gamma_p(q_1)\Gamma_p(q_2) \Gamma_p(k+1) \Gamma_p^3(p-k) \prod\limits_{j=1}^4 \Gamma_p(k+1-pr_j')} \\
    &= (-1)^{t_1+t_2+u_i+\epsilon_i} p \prod_{j=1}^4 \frac{(r_j)_k}{(q_j)_k} \frac{\Gamma_p(k+1+pq_1') \Gamma_p^3(k+1-p) \Gamma_p(k+1-pq_1')}{\Gamma_p(k+1)\prod_{j=1}^4 \Gamma_p(k+1-pr_j')} \\
    &= (-1)^{t_1+t_2+u_i+\epsilon_i}p \prod_{j=1}^4 \frac{(r_j)_k}{(q_j)_k}(1+O(p)).
\end{align*}
We can compute each 
\[
    B^{(2)}_k = \frac{1}{2} \lim_{t \to -k} \frac{d^2}{dt^2} \left((t+k)^3 R_2(t)\right) \biggr\vert_{t=-k}
\]
explicitly using the residue formula.  The derivative can be written purely in terms of logarithmic derivatives of Pochhammer symbols, which we can then apply (\ref{eq:log-der}) and (\ref{eq:second-log-der}) to in order to obtain an expression in terms of $\Gamma_p, G_1$, and $G_2$.  The exact expression for $B^{(2)}_k$ is a sum over 36 distinct terms, and so we leave it out for readability.  Although the computation is necessarily much larger than that needed to determine $B^{(1)}_k$, many of the underlying ideas and techniques are the same.  Once expressed in terms of $\Gamma_p, G_1$, and $G_2$ we are able to utilize Lemma \ref{lem:p-gamma-properties}, Theorem \ref{thm:p-gamma-approx}, as well as the identities (\ref{eq:G1-reflection}), (\ref{eq:G2-reflection}), (\ref{eq:G1-to-G2-at-1}), and (\ref{eq:G1-to-G2-at-1}) that we have established for $\Gamma_p$, $G_1$, and $G_2$ to eventually reduce $B^{(2)}_k$ to
\begin{align*}
    B^{(2)}_k &= A^{(2)}_k \Biggr(O(p) + J_2(k; \mathbf{\alpha}, \mathbf{\beta})  \\
    &\hspace{10mm} + \frac{1}{p}\left(\left(\sum_{j=1}^4\frac{\nu(k, t_j)}{r_j'} + \frac{\nu(u_1+1, k)}{q_1'} - \frac{\nu(k, u_2)}{q_1'}\right) J_1(k) + O(p) \right) \\
    &\hspace{10mm} + \frac{1}{p^2} \left(2\frac{\nu(k,u_2)}{(q_2')^2} + 2 \!\!\! \sum_{1 \leq i < j \leq 4} \!\!\! \frac{\nu(k, t_j)}{r_i'r_j'} - 2\sum_{j=1}^4 \frac{\nu(k,u_2)}{r_j'q_2'} \right) \Biggr).
\end{align*}
Note that the term being multiplied by $p^2$ is equal to zero for all $0 \leq k \leq t_2$.  We now multiply our residue sum by $(-1)^{t_1+t_2+u_i+\epsilon_i}$ which yields
\[
    0 = (-1)^{t_1+t_2+u_i+\epsilon_i}p^2 \left( \sum_{k=0}^{p-1} B^{(2)}_k + \sum_{k=1}^{u_1+1} D^{(2)}_k\right). 
\]
By Lemma \ref{lem:Dk-valuation} we have $p^2 D^{(2)}_k \equiv 0 \pmod{p^3}$ for all $1 \leq k \leq u_1+1$.  Additionally, by Corollary \ref{cor:hypergeometric-valuations-through-p} the $p$-adic valuation of the hypergeometric coefficients $A(k)$ is at least $2$ for all $t_2 \leq k < p-1$.  It follows that $p^2 B^{(2)}_k \equiv 0 \pmod{p^3}$ for all such $k$.  Therefore, 
\begin{align*}
    0 &\equiv (-1)^{t_1+t_2+u_i+\epsilon_i}p^3 \sum_{k=0}^{t_2} \tilde{B}_k \\
    &\equiv p^2 \Biggr[ \sum_{k=0}^{u_1} \prod_{j=1}^4 \frac{(r_j)_k}{(q_j)_k} \frac{J_1(k)}{q_1'} + \sum_{k=u_1+1}^{t_1} p\prod_{j=1}^4 \frac{(r_j)_k}{(q_j)_k} J_2(k) \\
    &\hspace{5mm} + \sum_{k=t_1+1}^{t_2} \prod_{j=1}^4 \frac{(r_j)_k}{(q_j)_k} \frac{J_1(k)}{r_1'} \Biggr] \pmod{p^3}.
\end{align*}
This final expression is exactly $C_2$, and so $C_2 \equiv 0 \pmod{p^3}$.  This completes the proof of Theorem \ref{thm:dwork-cong}.
\end{proof}

\Urlmuskip=0mu plus 1mu\relax
\RaggedRight
\printbibliography

@article{AO,
  title={A Gaussian hypergeometric series evaluation and Ap{\'e}ry number congruences},
  author={Scott Ahlgren and Ken Ono},
  journal={Crelle's Journal},
  year={2000},
  volume={2000},
  pages={187-212}
}

@article {BCM,
    AUTHOR = {Beukers, Frits and Cohen, Henri and Mellit, Anton},
    Title = {Finite hypergeometric functions},
    JOURNAL = {Pure Appl. Math. Q.},
    VOLUME = {11},
    YEAR = {2015},
    NUMBER = {4},
    PAGES = {559-589},
}

@article{CLZ,
  title={A supercongruence motivated by the Legendre family of elliptic curves},
  author={Heng Huat Chan and Ling Long and Wadim Zudilin},
  journal={Mathematical Notes},
  year={2010},
  volume={88},
  pages={599-602}
}

@article{Chu,
    title={A binomial coefficient identity associated with Beukers’ conjecture on Ap\'{e}ry numbers},
    author={Wenchang Chu},
    journal={Electronic Jounral of Combinatorics},
    year={2004},
    volume={11},
    issue={1},
    number={N15},
    url={https://doi.org/10.37236/1856}
}

@InProceedings{Diamond,
author="Diamond, Jack",
editor="Chudnovsky, David V.
and Chudnovsky, Gregory V.
and Cohn, Harvey
and Nathanson, Melvin B.",
title="p-adic gamma functions and their applications",
booktitle="Number Theory",
year="1984",
publisher="Springer Berlin Heidelberg",
address="Berlin, Heidelberg",
pages="168--175",
isbn="978-3-540-38788-6"
}

@article{Dwork,
author = {Bernard Dwork},
title = {$p$-adic cycles},
journal = {Publications Mathe\'{e}matiques de l'Institut des Hautes Scientifiques},
Volume = {37},
pages = {27-115},
year = {1969},
doi = {10.1007/BF02684886},
url = {https://doi.org/10.1007/BF02684886},
}

@article{FLRST,
    AUTHOR = {Fuselier, Jenny and Long, Ling and Ramakrishna, Ravi and Swisher, Holly and Tu, Fang-Ting},
    TITLE = {Hypergeometric functions over finite fields},
    JOURNAL = {Memoirs of the American Mathematical Society},
    YEAR = {to appear},
    PAGES = {1-112},
}

@article{FM,
author = {Fuselier, Jenny and McCarthy, Dermot},
year = {2014},
month = {07},
pages = {},
title = {Hypergeometric type identities in the $p$-adic setting and modular forms},
volume = {144},
journal = {Proceedings of the American Mathematical Society},
doi = {10.1090/proc/12837}
}

@article{Greene,
		author = "John Greene",
        title = "Hypergeometric Functions over Finite Fields",
        journal = "Transactions of the American Mathematical Society",
        volume = "301",
        year = "1987",
        pages = "77-101"
}

@book{Katz90,
    ISBN = {9780691085982},
    URL = {http://www.jstor.org/stable/j.ctt1bd6m1x},
    author = {Nicholas M. Katz},
    publisher = {Princeton University Press},
    title = {Exponential Sums and Differential Equations. (AM-124)},
    year = {1990}
}

@article{Katz09,
  title={Another Look at the Dwork Family},
  author={Nicholas M. Katz},
  journal={Algebra, arithmetic, and geometry: in honor of Yu. I. Manin. Vol II},
  pages={89-126},
  publisher={Birkh\"{a}user Boston, Inc.},
  location={Boston, MA},
  year={2009}
}

@article{Kilbourn,
  title={An extension of the Ap{\'e}ry number supercongruence},
  author={Timothy Kilbourn},
  journal={Acta Arithmetica},
  year={2006},
  volume={123},
  pages={335-348}
}

@misc{LLT,
      title={A Whipple $_7F_6$ formula revisited}, 
      author={Wen-Ching Winnie Li and Ling Long and Fang-Ting Tu},
      year={2021},
      eprint={2103.08858},
      archivePrefix={arXiv},
      primaryClass={math.NT}
}

@misc{LMFDB,
  key    = {LMFDB},
  author       = {The {LMFDB Collaboration}},
  title        =  {The L-functions and Modular Forms Database},
  howpublished = {\url{http://www.lmfdb.org}},
  year         = {2013},
  note         = {[Online; accessed 9 November 2021]},
}

@article {Long11,
    AUTHOR = {Long, Ling},
     TITLE = {Hypergeometric evaluation identities and supercongruences},
   JOURNAL = {Pacific J. Math.},
  FJOURNAL = {Pacific Journal of Mathematics},
    VOLUME = {249},
      YEAR = {2011},
    NUMBER = {2},
     PAGES = {405--418},
      ISSN = {0030-8730},
   MRCLASS = {33C20 (11B65)},
  MRNUMBER = {2782677},
MRREVIEWER = {Robert B. Osburn},
       DOI = {10.2140/pjm.2011.249.405},
       URL = {https://doi.org/10.2140/pjm.2011.249.405},
}

@article{Long20,
    title={Some Numeric Hypergeometric Supercongruences},
    author={Ling Long},
    year={2020},
    journal={Contemporary Mathematics},
    volume={753},
    pages={139-156},
    url={https://doi.org/10.1090/conm/753/15169}
}

@article{LR,
title = "Some supercongruences occurring in truncated hypergeometric series",
journal = "Advances in Mathematics",
volume = "290",
pages = "773 - 808",
year = "2016",
issn = "0001-8708",
doi = "https://doi.org/10.1016/j.aim.2015.11.043",
url = "http://www.sciencedirect.com/science/article/pii/S0001870815005216",
author = "Ling Long and Ravi Ramakrishna"
}

@article{LTYZ,
    title={Supercongruences Occured to Rigid Hypergeometric Type Calabi--Yau Threefolds},
    author={Ling Long and Fang-Ting Tu and Noriko Yui and Wadim Zudilin},
    year={2019},
    journal={MATRIX Annals.  MATRIX Book Series},
    volume={2},
    publisher={springer},
    url={https://doi.org/10.1007/978-3-030-04161-8_37}
}

@article{McCarthy,
 ISSN = {00029939, 10886826},
 URL = {http://www.jstor.org/stable/41505695},
 author = {Dermot McCarthy},
 journal = {Proceedings of the American Mathematical Society},
 number = {7},
 pages = {2241--2254},
 publisher = {American Mathematical Society},
 title = {On a supercongruence conjecture of rodriguez-villegas},
 volume = {140},
 year = {2012}
}

@article {MO,
    AUTHOR = {McCarthy, Dermot and Osburn, Robert},
     TITLE = {A {$p$}-adic analogue of a formula of {R}amanujan},
   JOURNAL = {Arch. Math. (Basel)},
  FJOURNAL = {Archiv der Mathematik},
    VOLUME = {91},
      YEAR = {2008},
    NUMBER = {6},
     PAGES = {492--504},
      ISSN = {0003-889X},
   MRCLASS = {11S80 (33C20)},
  MRNUMBER = {2465868},
MRREVIEWER = {Scott Ahlgren},
       DOI = {10.1007/s00013-008-2828-0},
       URL = {https://doi.org/10.1007/s00013-008-2828-0},
}

@article{Morita,
author = {Yasuo Morita},
title = {A p-adic analogue of the $\Gamma$ function},
journal = {Journal of the Faculty of Science, University of Tokyo, Section IA. Mathematics},
volume = {22},
issue = {2},
pages = {255-266},
}

@article {Mortenson,
    AUTHOR = {Mortenson, Eric},
     TITLE = {A {$p$}-adic supercongruence conjecture of van {H}amme},
   JOURNAL = {Proc. Amer. Math. Soc.},
  FJOURNAL = {Proceedings of the American Mathematical Society},
    VOLUME = {136},
      YEAR = {2008},
    NUMBER = {12},
     PAGES = {4321--4328},
      ISSN = {0002-9939},
   MRCLASS = {11S80 (33C20 33E50)},
  MRNUMBER = {2431046},
MRREVIEWER = {Scott Ahlgren},
       DOI = {10.1090/S0002-9939-08-09389-1},
       URL = {https://doi.org/10.1090/S0002-9939-08-09389-1},
}

@article{Ono98,
		author = "Ken Ono",
        title = "Values of Gaussian Hypergeometric Series",
        journal = "Transactions of the American Mathematical Society",
        volume = "350",
        year = "1998",
        pages = "1205-1223"
}

@article{OSZ,
     author = {Osburn, Robert and Straub, Armin and Zudilin, Wadim},
     title = {A modular supercongruence for $_6F_5$: An Ap\'ery-like~story},
     journal = {Annales de l'Institut Fourier},
     pages = {1987--2004},
     publisher = {Association des Annales de l{\textquoteright}institut Fourier},
     volume = {68},
     number = {5},
     year = {2018},
     doi = {10.5802/aif.3201},
     language = {en},
     url = {http://www.numdam.org/articles/10.5802/aif.3201/}
}

@article{Swisher,
    author={Swisher, Holly},
    year={2015},
    title={On the supercongruence conjectures of van Hamme},
    journal={Research in the Mathematical Sciences},
    number={18},
    volume={2},
    url={https://doi.org/10.1186/s40687-015-0037-6},
    doi={10.1186/s40687-015-0037-6}
}

@incollection {vH,
    AUTHOR = {van Hamme, Lucien},
     TITLE = {Some conjectures concerning partial sums of generalized
              hypergeometric series},
 BOOKTITLE = {{$p$}-adic functional analysis ({N}ijmegen, 1996)},
    SERIES = {Lecture Notes in Pure and Appl. Math.},
    VOLUME = {192},
     PAGES = {223--236},
 PUBLISHER = {Dekker, New York},
      YEAR = {1997},
   MRCLASS = {33C20 (11S80)},
  MRNUMBER = {1459212},
MRREVIEWER = {Anne Schilling},
}

@misc{Watkins,
    title={Hypergeometric motives over Q and their L-functions},
    author={Mark Watkins},
    note={preprint at http://magma.maths.usyd.edu.au/~watkins/papers/known.pdf}
}

@article {Zudilin,
    AUTHOR = {Zudilin, Wadim},
     TITLE = {Ramanujan-type supercongruences},
   JOURNAL = {J. Number Theory},
  FJOURNAL = {Journal of Number Theory},
    VOLUME = {129},
      YEAR = {2009},
    NUMBER = {8},
     PAGES = {1848--1857},
      ISSN = {0022-314X},
   MRCLASS = {33C20 (11B65 11F33 11Y55 11Y60)},
  MRNUMBER = {2522708},
MRREVIEWER = {David Y. Jao},
       DOI = {10.1016/j.jnt.2009.01.013},
       URL = {https://doi.org/10.1016/j.jnt.2009.01.013},
}
\end{document}